\documentclass[12pt]{article}

\usepackage[utf8]{inputenc}
\usepackage[english]{babel}
\usepackage{amsthm, amsmath, amssymb}
\usepackage{calrsfs} % curly mathcal
\usepackage{graphicx,color}
\usepackage{soul} % strikeout
\usepackage[margin=1in]{geometry}
\usepackage{hyperref,cleveref}
\newcommand{\m}[1]{\texorpdfstring{$#1$}{}}

\newtheorem{lemma}{Lemma}

\newtheorem{corollary}{Corollary}

\newtheorem{proposition}{Proposition}
\newtheorem{theorem}{Theorem}
\newtheorem{remark}{Remark}

\newcommand{\E}{\mathbb{E}}
\renewcommand{\P}{\mathbb{P}}
\newcommand{\R}{\mathbb{R}}
\newcommand{\Cov}{\mathrm{Cov}}
\renewcommand{\a}{\alpha}
\newcommand{\1}{\mathbf{1}}

\title{Gaussian consensus processes and their Lyapunov exponents}
\date{13 March 2026}

\begin{document}

\author{Edward Crane\footnote{Heilbronn Institute for Mathematical Research and School of Mathematics, University of Bristol, Bristol BS8~1UG, United Kingdom}
\ and Stanislav Volkov\footnote{Centre for Mathematical Sciences, Lund University, Lund 22100, Sweden}}

\maketitle

\abstract{We introduce a simple dynamic model of opinion formation, in which a finite population of individuals hold vector-valued opinions. At each time step, each individual's opinion moves towards the mean opinion but is then perturbed independently by a centred multivariate Gaussian random variable, with covariance proportional to the covariance matrix of the opinions of the population.  We establish precise necessary and sufficient conditions on the parameters of the model, under which all opinions converge to a common limiting value. Asymptotically perfect correlation emerges between opinions on different topics. Our results are rigorous and based on properties of the partial products of an i.i.d.~sequence of random matrices. Each matrix is a fixed linear combination of the identity matrix and a real Ginibre matrix. We derive an analytic expression for the maximal Lyapunov exponent of this product sequence.  We also analyse a continuous-time analogue of our model.
}



\section{Introduction}
In this paper, we introduce two simple stochastic processes which we call \emph{Gaussian consensus processes}. Model~A is a discrete-time Markov process, and Model~B is a continuous-time Markov diffusion. They are fairly reasonable and yet highly mathematically tractable stochastic models of opinion dynamics. 

Opinion dynamics is a popular subject in the current mathematical modelling and social science literature, for example, see Hassani et al.~\cite{fusion}, Peralta et al.~\cite{PKI} and references therein. Often the focus is on explanatory power; for example, the recent survey by Devia et al.~\cite{Devia} compares the qualitative features of the output of a variety of opinion formation models with opinion data from a selection of real-world surveys. Our focus, however, will be on the theoretical probabilistic properties of a stochastic opinion formation model that is designed to be highly tractable, but nevertheless displays qualitative features which match at least one familiar aspect of opinion dynamics in the real world. {Another famous opinion dynamics model, the \emph{voter model}, concerns discrete opinions evolving on a graph, with individuals being influenced by their immediate neighbours. This is not closely related to the topic of the present paper, because our models are mean field with opinions in a continuum.} 

Let us mention a few stochastic opinion dynamics models which are somewhat closely related to the current paper. In the Weisbuch-Deffuant model \cite{DWmodel}, $N$ individuals each hold an opinion in the interval $[0,1]$; at each time step two individuals are chosen uniformly at random, and if their opinions are less than~$\epsilon$ apart then they each move their opinion towards the mean of their two opinions. The model converges almost surely to a limit in which clusters of individuals all agree perfectly within each cluster, and the clusters are separated by distance at least~$\epsilon$. In a noisy variation considered by Pi\~{n}eda et al.~\cite{Pineda}, each individual occasionally resamples their opinion from the uniform distribution on $[0,1]$. The \emph{bounded confidence} cutoff, controlled by the parameter~$\epsilon$, is a nonlinearity which induces clustering of the opinions into multiple clusters when the non-linearity is strong enough in comparison to the noise. 

In the well-known DeGroot model, individuals correspond to vertices of a graph. Edges correspond to pairs of individuals who influence each other. At each time step, each individual deterministically updates their $\mathbb{R}$-valued opinion to the mean of their neighbours' opinions at the previous time step. Stern and Livan \cite{SternLivan} recently studied a stochastic variant of the DeGroot model in which the opinions are further modified at each time step by the addition of Gaussian noise. Stern and Livan's equation~(3.4) allows the variance of the noise terms to depend on the current state in a general way, which includes the one-dimensional case of our Model~A as a special case. However, their main object of study is a model in which the variance of the noise term is constant.  In the alternative models considered in Section~7 of~\cite{SternLivan}, the variance of the noise term varies from individual to individual according to how far their opinion is from the empirical mean. Stern and Livan's main focus is a numerical study of the effect of the community structure of the graph which governs the interactions. In contrast, in our models, the opinions are in $\mathbb{R}^d$, and the covariance of the noise term is a function of the empirical covariance of the opinions, while the underlying graph is the complete graph on $N$ vertices. 

{Finally, we would like to mention a few very recent papers related to opinion dynamics and somewhat relevant to our paper: \cite{OD3,OD2,OD1}.}

\subsection{Models}

The models that we introduce and study in this paper are \emph{mean field} (the underlying graph is the complete graph) and {\em linear} (there is no bounded confidence cutoff). 
%This makes them probably less realistic as direct models of opinion dynamics in the real world than some of the models mentioned above\footnote{Our models could potentially arise as linearizations when studying the stability of a fixed point of a deterministic nonlinear opinion dynamics model when only a small amount of stochastic noise is added; however we will treat them as interesting stochastic processes in their own right.}.
The novel feature of our models is the combination of Gaussian noise and invariance under affine linear changes of coordinates; these two properties allow us to compute {\em precisely} the critical parameters on the boundary between stability and instability, using the classical theory of Lyapunov exponents of random matrix products.

In Model~A, $N \ge 2$ individuals each hold real-valued opinions on $d$ topics. Each individual updates their opinions once every day. The opinions of individual $i$ on day $t$ are represented by a row vector $X_i(t) \in \mathbb{R}^d$. We collect these row vectors in an $N\times d$ matrix $X(t)$ whose $i^{th}$ row is $X_i(t)$. The initial condition $X(0)$ may be deterministic or random. For each time $t \in \mathbb{N}_0$, we denote by $\mu^X(t)$ and $\Cov^X(t)$ the empirical mean and empirical covariance matrix of the $N$ opinion vectors at time $t$. The model has two parameters: $\alpha > 0$ and $\beta \in [0,1]$. We may think of $\alpha$ as a measure of how much the individuals like to stand out from the crowd, and $\beta$ as a measure of their self-confidence.   For each $t \ge 1$, conditional on $X(t-1)$, let $Y_1(t), \dots, Y_N(t)$ be independent samples from the multivariate Gaussian distribution with mean $\mu^X(t-1)$ and $d\times d$ covariance matrix $\alpha\, \Cov^X(t-1)$. Then for each $1 \le i \le N$, let 
\begin{equation}\label{E: first definition of model A} 
X_i(t) = \beta X_{i}(t-1) + (1-\beta) Y_i(t)\,.
\end{equation}
Note that conditional on $X(t-1)$ the  $N$ opinion vectors $X_1(t), \dots, X_N(t)$ are independent. Also, note that  $\left(X(t)\right)_{t \in \mathbb{N}_0}$ is a discrete-time Markov process.

Formally, 
%Recall that $\mu^X(t) \in \mathbb{R}^d$ denotes the empirical mean opinion at time $t$:
$$ 
\mu^X(t) = \frac{1}{N}\sum_{k=1}^N X_k(t)\,,
$$
and the coordinates of $\mu^X(t)$ are $\mu_1^X(t),\mu_2^X(t),\dots,\mu_d^X(t)$. Consider the centred $N \times d$ matrix~$\bar X(t)$ with rows $X_i(t) - \mu^X(t)$ for $i = 1, \dots, N$. Each column of $\bar X(t)$ has sum zero. Define $$\Cov^X(t) = \frac{1}{N} \bar X(t)^T \bar X(t)\,.$$ Thus $\Cov^X(t)$ is the $d\times d$ covariance matrix of the empirical distribution at time $t$. Its entry in row $i$ and column $j$ is
$$ 
\Cov_{i,j}^X(t) =  \frac{1}{N}\sum_{k=1}^N \left(X_{k,i}(t) - \mu_i^X(t) \right)\left(X_{k,j}(t) - \mu_j^X(t)\right)\,.
$$
We now describe the transition law of the Markov chain $\left(X(t)\right)_{t \ge 0}$. At each time step $t \ge 1$, conditional on $X(t-1)$, the opinion vectors $X_1(t), \dots, X_N(t)$ are independent and for each $i = 1, \dots, N$, $X_i(t)$ has the multivariate Gaussian distribution with mean
$$ 
\beta X_i(t-1) + (1-\beta)\mu^X(t-1)$$ and covariance 
$$
\rho\ \Cov^X(t-1)\,,$$
 where 
\begin{align}\label{defrho}
\rho :=\alpha (1-\beta)^2\,.    
\end{align}
%One can easily verify that this description aligns with the one provided earlier in the introduction.  

Model~A is invariant under affine linear transformations of the opinion space: if $f: x \mapsto xU + v$, where $U$ is any $d\times d$ real matrix and $v \in \mathbb{R}^{d}$, then the sequence $\left(X'(t)\right)_{t \ge 0}$ defined by $X'_i(t) = f(X_i(t))$ is a Markov process with the same transition kernel as $\left(X(t)\right)_{t \ge 0}$. 

We will show in Proposition~\ref{prop: Covs are Markov} that $\left(\mathrm{Cov}^X(t)\right)_{t \ge 0}$ is a Markov chain on its own, and conditional on its entire evolution, the empirical mean process $\left(\mu^X(t)\right)_{t \ge 0}$ is a martingale with independent Gaussian increments. This structure enables us to reduce questions about the model to questions about the sequence of partial products of a sequence $M(1), M(2), \dots$ of i.i.d.~random $(N-1)\times(N-1)$ matrices. Each matrix $M(t)$ is distributed as a linear combination of the identity and a \emph{real Ginibre matrix}. A random $k\times k$ matrix is called a real Ginibre matrix if its entries are i.i.d.~standard normal random variables. Its distribution is called the real Ginibre ensemble of rank~$k$. A precise description of $M(t)$ will be given below in Section~\ref{SS: results about Model~A}.

An important question about any opinion dynamics model is whether the opinions converge almost surely to a consensus, i.e.~to a finite limit in the opinion space $\mathbb{R}^d$. Theorem~\ref{t1} answers this question for all values of~$N, \alpha$, and $\beta$. Our analysis reduces the problem to determining the sign of the maximal Lyapunov exponent of the matrix product sequence.  To keep the paper self-contained, we include a brief summary of the key definitions and theorems about Lyapunov exponents of i.i.d.~random matrix products.  Extending a classical result of Newman \cite{Newman} about products of real Ginibre matrices, we use the fact that the law of~$M(1)$ is invariant under orthogonal conjugation to compute the maximal Lyapunov exponent of our i.i.d.~matrix product exactly. For even $N \ge 4$ we express it as an infinite series involving the digamma function, and for odd $N$ we express it in terms of the exponential integral function $\mathrm{Ei}_1$. We believe this computation has not previously appeared in the literature. In fact there are only a few random matrix ensembles with nontrivial maximal Lyapunov exponents that are known to be expressible without reference to an inexplicit invariant measure on a projective space; see for example \cite{Forrester, Kargin, Marklof, Newman}.  Our computation of the maximal Lyapunov exponent determines the precise region in parameter space where consensus is reached almost surely. We show that outside of this region the opinion vectors almost surely diverge as $t \to \infty$. In the critical case, on the boundary of the consensus region, almost surely the diameter of the set of opinions is neither bounded nor bounded away from $0$ as $t \to \infty$. In the supercritical case, the individual opinions and the mean all tend to $\infty$ almost surely.

We also obtain an expression for all $N-1$ Lyapunov exponents, and compute explicit positive lower bounds on the gaps between successive Lyapunov exponents, in terms of the parameters $(N, \alpha, \beta)$.  The gap between the first two Lyapunov exponents has an interpretation in the opinion dynamics, given in Theorem~\ref{t2}. The gap is the exponential rate at which the population's opinions on the $d$ different topics tend towards being perfectly correlated. It is striking that even in a very simple linear model we can observe the formation of a one-dimensional `political spectrum', whose alignment in the $d$-dimensional opinion space is random. This is described in Theorem~\ref{t3}.

Model~B is a continuous-time analogue of Model~A, with time running over $[0,\infty)$. {It is motivated by considering the possible scaling limit of Model~A as the updates become more frequent and less strong, though we do not prove such a scaling limit result in this paper; we pose it as an open problem to identify the correct scaling regime.} 
It is defined in detail in Section~\ref{SS: Model~B description} below. Briefly, the $N$ opinion vectors diffuse, solving a stochastic differential equation. The diffusion coefficient of each individual's opinion is defined to be the unique non-negative definite symmetric square root of the empirical covariance matrix of the $N$ opinion vectors in $\mathbb{R}^d$. There is also a linear drift towards the empirical mean of the $N$ opinion vectors. We check that the SDE is well-posed using the Araki--Yamagami inequality. We determine the critical strength of the linear drift term, above which consensus is reached almost surely, as a function of $N$.  The random matrix product that occurs in the analysis of Model~A is replaced in the analysis of Model~B by the ellipsoid diffusion process described by Norris, Rogers and Williams \cite{NRW}, who gave an elementary treatment of earlier work by Dynkin and Orihara.

\subsection{Results about the discrete time model (Model~A)}\label{SS: results about Model~A}

Before stating our results, we rule out two special cases. Firstly we assume $\beta \neq 1$, since the case $\beta = 1$ corresponds to completely stubborn individuals who never change their opinion. Thus $\beta \in [0,1)$. Secondly, we assume that almost surely there is no nontrivial linear combination of the $d$ opinions on which all $N$ individuals agree. In other words, we assume that $\Cov^X(0)$ is positive definite almost surely. For if we start with perfect consensus on some `mixture of topics' represented by a linear functional of the $d$ coordinates, then no individual can ever change their opinion on this mixture; we may as well ignore this direction, which has the effect of projecting the model to a lower dimension.  As a consequence, we must assume $N \ge d+1$.

\begin{theorem}[Phase transition]\label{t1}
For any real $m > 0$ let $\phi_m$ be the entire function defined for all $x \in \mathbb{C}$ by
\begin{equation} \label{def of phi}
\phi_m(x) := e^{-x} \sum_{j=0}^\infty \frac{x^j}{j!} \psi(j+m)=\psi(m)+\int_0^1 \frac{(1-e^{-xs})(1-s)^{m-1}}{s}\mathrm{d}s\,,
\end{equation}
where $\psi$ is the digamma function. Let $z = \frac{N\beta^2}{2\alpha(1 -\beta)^2}=\frac{N\beta^2}{2\rho}$, and define
$$ 
\lambda_1=\lambda_1(\alpha, \beta, N) = \frac{1}{2}\left[ \phi_{{\frac{N-1}{2}}}
\left(\frac{N\beta^2}{2\rho}\right)
+\log \frac{2\rho}N \right] = \log \beta + \frac{1}{2}\left[ \phi_{{\frac{N-1}{2}}}(z) - \log z \right]
\,,
$$ (where the final expression applies only in the case $\beta > 0$.)
\begin{enumerate}
\item[(i)] If $\lambda_1 < 0$ then  almost surely there exists a random limit $\mu_\infty \in \mathbb{R}^d$ such that for all $i=1, \dots, N$, $X_i(t) \to \mu_\infty$ as $t \to \infty$.
\item[(ii)] If $\lambda_1 = 0$ then the mean $\mu^X(t)$  and each of the individual opinions $X_i(t)$ diverge as $t \to \infty$, almost surely. 
\item[(iii)] If $\lambda_1 > 0$ then for every deterministic row vector $u \in \mathbb{R}^d \setminus\{0\}$, almost surely for all $i=1, \dots, N$, $\lim_{t\to\infty}| u \cdot X_i(t)|  = \infty$ and $\lim_{t \to \infty} |u \cdot \mu^X(t)| = \infty$.
\end{enumerate}
\end{theorem}
{Some illustrations of how $\lambda_1$ depends on $\alpha$, $\beta$ and $N$, and relevant phase transitions, are given in Figures~\ref{fig1} and~\ref{fig2}.}

\begin{figure}
\centering
\begin{minipage}{.5\textwidth}
  \centering
  \includegraphics[width=.9\linewidth]{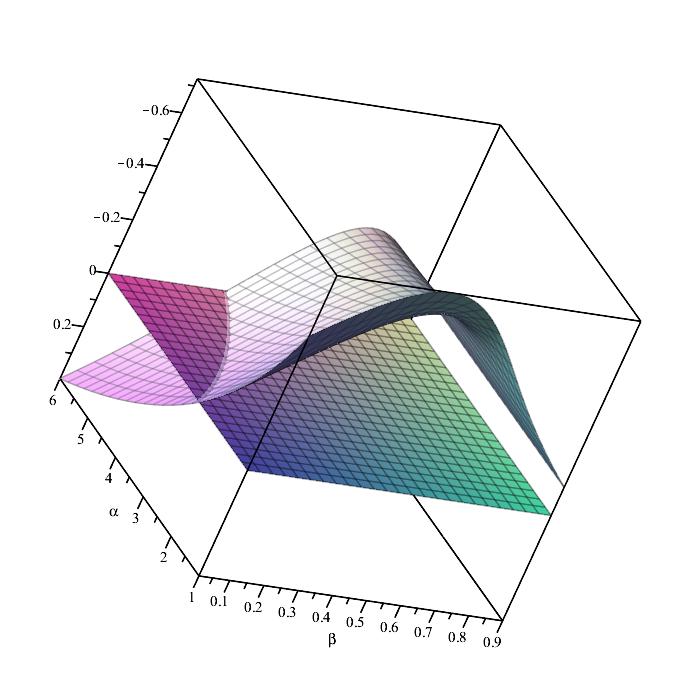}
  \end{minipage}%
\begin{minipage}{.5\textwidth}
  \centering
  \includegraphics[width=.9\linewidth]{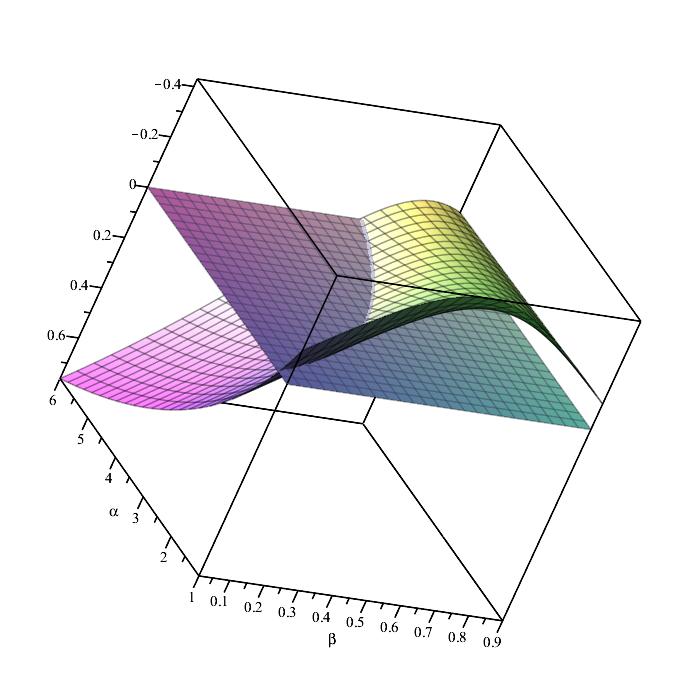}
\end{minipage}%
\caption{$\lambda_1$ as function of $\alpha,\beta$ for $N=3$ (left) and $N=9$ (right).}
\label{fig1}
\end{figure}

\begin{figure}
%\vspace{2mm}
\centering
\begin{minipage}{.49\textwidth}
  \centering
  \includegraphics[width=0.75\linewidth]{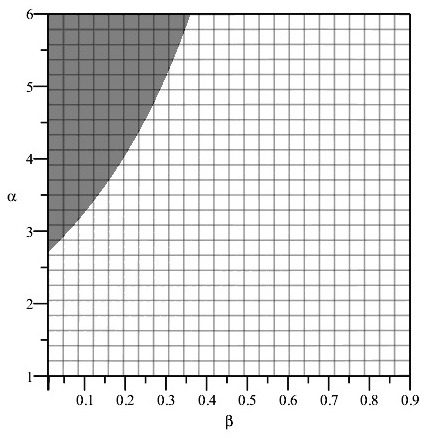}
\end{minipage}
\begin{minipage}{.49\textwidth}
  \centering
  \includegraphics[width=0.75\linewidth]{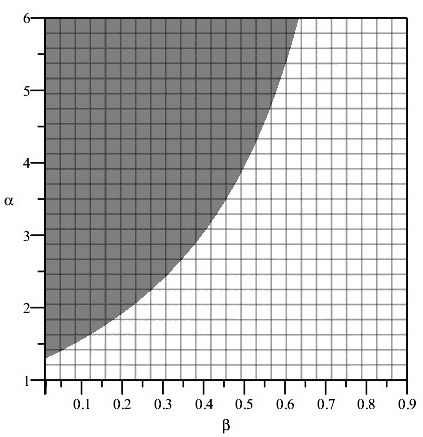}
\end{minipage}
\caption{Phase transition of $\lambda_1$ for $N=3$ (left) and $N=9$ (right). The grey area is where the model is unstable, i.e.\ $\lambda_1\ge 0$.}
  \label{fig2}
\end{figure}

\begin{remark}
The function $\phi_m$ is defined by the left-hand equality in equation~\eqref{def of phi} and we will prove in Lemma~\ref{lem:phi} that it is indeed equal to the final expression in~\eqref{def of phi}.
\end{remark}
\begin{remark}
 We show that as $N\to\infty$ with $\alpha > 0$ and $\beta \in [0,1)$ fixed, we have $\lambda_1\to \frac12\log(\alpha(1-\beta)^2+\beta^2)$. See  Lemma~\ref{lem:lambda1asympt}.
\end{remark}
\begin{remark}
We leave it as an open problem whether the conclusion of part (iii) holds in the case where $\lambda_1(\alpha, \beta, N) = 0$. {Namely, we do not know if the divergence of the opinions occurs in {\em all} directions, or rather only in some subspace of a lower dimension.}
\end{remark}
\begin{remark}
In Proposition~\ref{prop:Nodd} in Section~\ref{S: some properties} we will give closed form analytic expressions for $\lambda_1(\alpha, \beta, N)$ in the cases where $N$ is odd.
\end{remark}

For $t=0,1,2,\dots$ let
\begin{equation}\label{E: M definition}
M(t) = \sqrt{\frac{\rho}{N}} G^{(N-1)}\!(t) + \beta I_{N-1}\,,
\end{equation}
where $\left(G^{(N-1)}\!(t)\right)_{t \ge 0}$ is an i.i.d.~sequence drawn from the rank $(N-1)$ real Ginibre ensemble and $I_{N-1}$ is the $(N-1)\times(N-1)$ identity matrix. The quantity $\lambda_1$ appearing in Theorem~\ref{t1} is in fact the largest Lyapunov exponent of the product sequence $\left(P(t)\right)_{t \ge 1}$ defined by 
$$
P(t)=M(t) \dots M(2) M(1).
$$ 
Indeed, by Section~\ref{Lyapexp}, the limit ${\lambda_1} =\lim_{t\to\infty} \frac1t\log\|P(t)\|$ exists; Theorem~\ref{t1} shows $\lambda_1<0$ iff Model~A reaches consensus a.s., with the critical and supercritical cases treated there.

The next statement describes {\it all} Lyapunov exponents, not just the largest one, including the gap between the consecutive ones.
\begin{theorem}[Properties of the Lyapunov exponents]\label{t2} {\quad}
\begin{itemize}
\item[(a)] For parameters $N, \alpha, \beta$, and $\rho = \alpha(1-\beta)^2$, for each $k=1, \dots, N-1$, the $k^{th}$ Lyapunov exponent $\lambda_k$ of the i.i.d.~random matrix product sequence $\left(P(t)\right)_{t \ge 1}$ is given by 
$$ 
\lambda_k=\lambda_k(\alpha,\beta,N)=\frac{1}{2}\mathbb{E}\left(\log\left(\left(\beta W_k + \frac{\sqrt{\rho}}{\sqrt{N}}X_1\right)^2 + \frac{\rho}{N}\sum_{i=2}^{N-k} X_i^2\right)\right)\,,
$$ 
where $W_k$ is a random variable distributed as the Euclidean length of the projection of the unit vector $e_{k}$ onto the orthogonal complement of the span of the first $k-1$ rows of the matrix $M(1)$, and $X_1, \dots, X_{N-k}$ are i.i.d.~standard Gaussians, independent of $W_k$.
\item[(b)]
The Lyapunov spectrum of $\left(P(t)\right)_{t \ge 1}$ is simple, meaning that the $N-1$ Lyapunov exponents $\lambda_1, \dots, \lambda_{N-1}$ are distinct; moreover
$$\lambda_{k}- \lambda_{k+1} \ge \frac{0.15}{(N-k)\left(\frac{N\beta^2}{\alpha(1-\beta)^2} + N-k\right)}\,,$$ for every $k$ in the range $1 \le k \le N-2$. (See Theorem~\ref{t3} for a stronger result about $\lambda_1 - \lambda_2$.)
\end{itemize}
\end{theorem}

{The interpretation of Theorem~\ref{t2} is that there is a random direction in $\mathbb{R}^d$ in which the opinions line up in the limit $t \to \infty$. The law of this direction is the unique invariant measure on $P(\mathbb{R}^d)$ under the action of the group of linear isometries of the quadratic form $\left(\bar{X}(0)^T\,\bar{X}(0)\right)^{-1}$. The convergence of the population of opinions to one-dimensionality occurs exponentially with rate $\lambda_1 - \lambda_2$; later we will show that $\lambda_1 - \lambda_2$ behaves like   $1/N$ in the case $\beta = 0$, and is at least  $c(\alpha, \beta)/N^2$ in the general case.   }

Our third result concerns a normalized variant of Model~A. Let $e_{max}(t)$ be the largest eigenvalue of the covariance matrix $\mathrm{Cov}^X(t)$. Then define a \emph{normalized} (i.e.~shifted and rescaled) matrix $\hat{X}(t)$ for $t = 0,1,2, \dots$ whose $i^{th}$ row is 
$$\hat{X}_i(t) = e_{max}(t)^{-1/2} \left(X_i(t) - \mu^X(t)\right)\,.$$
This normalized matrix has the properties that its rows sum to zero and the largest eigenvalue of the covariance matrix of the $N$ vectors $\hat{X}_1(t), \dots, \hat{X}_N(t)$ in $\mathbb{R}^d$ is $1$. These constraints mean that the normalized opinions can neither converge to a consensus nor diverge to infinity. The affine invariance of Model~A implies that $\left(\hat{X}(t)\right)_{t \ge 0}$ is a Markov chain on its own. Each transition may be described by applying the transition rule for Model~A to the matrix $\hat{X}(t)$ and then normalizing the resulting matrix.

To understand the dynamics of the normalized model, we are interested not in $\lambda_1$ but in the gap $\lambda_1 - \lambda_2$. In the normalized model, a randomly oriented one-dimensional political spectrum emerges at an exponential rate $\lambda_1 - \lambda_2$.

\begin{theorem}\label{t3}
As $t\to\infty$, $\Cov^X(t)/\Vert \Cov^X(t)\Vert$ converges to a random rank one matrix, regardless of the value of  $\lambda_1(\alpha, \beta, N)$. Equivalently, there is a random line $\ell$ through $0$ in $\mathbb{R}^d$ to which the normalized opinions $\hat{X}_1(t), \dots, \hat{X}_N(t)$ all converge a.s.~as $t \to \infty$. The exponential rate of convergence is $\lambda_1 - \lambda_2$, which satisfies
\begin{align}\label{eq:t3}
\lambda_1-\lambda_2 = \frac1{2N}
\left(1-\left[\frac{\beta^2}{\rho+\beta^2}\right]^2\right)(1+o(1))
\end{align}
as $N \to \infty$ with $\alpha$ and $\beta$ fixed such that $\rho = \alpha(1-\beta)^2 > 0$.
If $v$ denotes a unit vector parallel to $\ell$, then $(\hat{X}_1(t).v, \dots, \hat{X}_N(t).v)$ converges in distribution as $t \to \infty$ to a random variable that is uniformly distributed on the $(N-2)$-dimensional sphere 
$$ \mathbb{S}^{N-2} := \left\{ (x_1, \dots, x_N) \in \mathbb{R}^N \,:\, \sum_{i=1}^N x_i = 0,\;\;\sum_{i=1}^N x_i^2 = 1 \right\}\,.$$
\end{theorem}
%SENTENCE ADDED IN REVISION:

{
 It follows from Theorem~\ref{t3} that for any two topics, i.e. coordinates $1 \le j < k  \le d$, we have asymptotically perfect correlation in the people's opinions on those topics:
 \[\mathrm{Corr}(X_{\cdot,j}(t), X_{\cdot,k}(t)) \to \pm 1 \text{ as $t \to \infty$}.\]
The mechanism by which this occurs is that the products of the matrices that update the covariance matrix contract directions in projective space, at asymptotic rate $\lambda_1 - \lambda_2 > 0$, towards a random direction, yielding the rank one covariance matrix and perfect correlation in the limit.
}

Finally, we show that the sequence of covariance matrices itself is a Markov chain. Let $\mathcal{F}^X_t$ be the $\sigma$-algebra generated by $X_0, \dots, X_t$. 
\begin{proposition}\label{prop: Covs are Markov}
The sequence of covariance matrices $\left(\Cov^X(t)\right)_{t \ge 0}$ is a Markov process with respect to its own generated filtration. Moreover, for each $t \ge 1$ the conditional distribution of $\Cov^X(t)$ given $\mathcal{F}^X_{t-1}$ is a function of $\Cov^X(t-1)$ alone. We also have
\begin{align*}
\mathbb{E}(\mathrm{Cov}^X(t) \mid \mathcal{F}^X_{t-1}) &=  \mathbb{E}(\mathrm{Cov}^X(t) \mid X(t-1))= \left(\frac{N-1}{N}\rho + \beta^2\right)\mathrm{Cov}^X(t-1)  \quad \text{and}\\
\mathrm{Var}(\mathrm{Cov}_{ij}^X(t) \mid \mathcal{F}^X_{t-1}) &=  \mathrm{Var}(\mathrm{Cov}_{ij}^X(t) \mid X_{t-1})\\
& \hspace{-8mm}=  \left(\frac{\rho^2(N-1)}{N^2} + 2(\beta-1)^2\frac{\rho}{N}\right)\left(\mathrm{Cov}_{ii}^X(t-1)\, \mathrm{Cov}_{jj}^X(t-1)  + \mathrm{Cov}^X_{ij}(t-1)^2\right)\,.
\end{align*}
\end{proposition}
Fix an $N-1$ by $N$ matrix~$V$ whose rows are orthonormal and all orthogonal to the vector $(1, \dots, 1)$.
Then the second sentence of Proposition~\ref{prop: Covs are Markov} may be rephrased by saying that the mapping $X \to \frac{1}{N} X^T V^TV X$ from the space of real $N\times d$ matrices to the space of non-negative definite symmetric $d$ by $d$ matrices is a \emph{strong lumping} of the Markov process $\left(X(t)\right)_{t \ge 0}$.

\subsection{Continuous-time model (Model~B)}\label{SS: Model~B description}
Model~B is a natural continuous time analogue of Model~A. It has just two parameters: $N$, the number of individuals, and $\gamma \in \mathbb{R}$, which determines the drift of the opinion vectors towards their empirical mean.

Let $B^1, \dots, B^N$ be independent standard $d$-dimensional Brownian motions on the time interval $(0, \infty)$, thought of as row vectors, and let $B_{i,j}(t)$ denote the $j^{\text{th}}$ coordinate of $B^i$ at time $t$, so that $B(t)$ is an $N$ by $d$ matrix. {We continue to assume that $N \ge d+1$.} Let $Z(0)$ be a deterministic or random $N$ by $d$ real matrix whose rows represent vectors $Z_1(0), \dots, Z_N(0) \in \mathbb{R}^d$. We suppose that the initial vectors are not all contained in any affine hyperplane. Then define a continuous-time matrix-valued process $Z(t)_{t \in [0,\infty)}$ to be the unique strong solution of the stochastic differential equation 
\begin{equation}\label{E: Ito model}
d Z_i(t) = -\gamma\left(Z_i(t) - \mu^Z(t)\right) dt +  d B^i(t) T(t) \quad \text{for $i = 1, \dots, N$}\,,\end{equation}
where $$ \mu^Z(t) := \frac{1}{N} \sum_{i=1}^N Z_i(t)\,,$$
and $T(t)$ is the unique non-negative definite symmetric square root of the matrix $\Cov^Z(t)$ that is defined by  
$$ \Cov^Z_{a,b}(t) := \frac{1}{N}\sum_{i=1}^N \left(Z_{i}(t)_a - \mu^Z(t)_a \right)\left(Z_i(t)_b - \mu^Z(t)_b\right)\,.$$

Writing $\overline{Z}(t)$ for the $N$ by $d$ matrix whose $i^{\textup{th}}$ row is $Z_i(t) - \mu^Z(t)$, we have 
$$ 
\Cov^Z(t) = \frac{1}{N} \overline{Z}(t)^T\, \overline{Z}(t)\,.
$$
We can write the stochastic differential equation~\eqref{E: Ito model} fully in matrix form as
\begin{equation}\label{E: SDE} d Z(t) = -\gamma\overline{Z}(t) dt  +   d B(t) T(t)\,.\end{equation}
The existence and uniqueness of the strong solution of the stochastic differential equation~\eqref{E: SDE} will be checked in Lemma~\ref{L: SDE is well posed} in Section~\ref{SecB}.

In order to state our main theorem about Model~$B$, we need some notation.  Let $\mathbf{1}_{k,\ell}$ denote the $k$ by $\ell$ matrix whose entries are all equal to $1$.
\begin{theorem}\label{T: continuous time solution} Let $Z(0)$ be an $N$ by $d$ real matrix such that $V Z(0)$ has rank $d$. Let $G_t$ be the right-invariant Brownian motion on $GL(N-1,\mathbb{R})$ defined by the initial condition $G_0 = I_{N-1}$ and the  Stratonovich form matrix SDE
\begin{equation} \label{E: H def} 
\partial G(t) = \partial \hat{W}(t) G(t)\,,
\end{equation}
where $\hat{W}$ is an $(N-1)$ by $(N-1)$ matrix of independent standard Brownian motions. Then a strong solution of equation~\eqref{E: Ito model} with initial condition $Z(0)$ is obtained by letting
$$ \gamma' = \gamma + \frac{1}{2N}\,,$$
$$ \overline{Z}(0) = V^TVZ(0) = \left(I_N - \frac{1}{N}\mathbf{1}_{N,N}\right)Z(0)\,,$$
$$ \overline{Z}(t) = e^{-\gamma' t} V^T G(t/N) V \overline{Z}(0)\,, \quad \text{and}$$
$$ Z(t) = \overline{Z}(t) +  \frac{1}{N}\mathbf{1}_{N,N}Z(0) + \frac{1}{\sqrt{N}} \mathbf{1}_{N,1} \int_0^t dF(t) T(t) \,,$$
where $F$ is a $1$ by $d$ matrix of standard Brownian motions independent of $\hat{W}$, and $T(t)$ is the unique non-negative definite symmetric square root of $\frac{1}{N} \overline{Z}(t)^T\overline{Z}(t)$.
\end{theorem}
We remark that the right-invariant Brownian motion $G$ has nothing to do with the Ginibre matrices $G^{(N-1)}(t)$ that we use in the analysis of Model~A. Rather, the notation is chosen to be consistent with
Norris, Rogers, and Williams~\cite{NRW}, who showed that the  process $Y(t) : = G(t)^T G(t)$ is a Markov process whose characteristic exponents\footnote{The characteristic exponents of $Y(\cdot)$ are the values $\lim_{t \to \infty} \frac{1}{t} \log e_i(t)$ where $e_1(t) > \dots > e_{N-1}(t)$ are the sorted eigenvalues of $Y(t)$.} are $N - 2i$ for $i = 1,2, \dots, N-1$, and moreover the orthonormal frame of eigenvectors of $Y(t)$ converges a.s.~to a random limit in $O(N-1)$ as $t \to \infty$.  (See also \cite{Newman, LeJan}, which appeared around the same time as \cite{NRW} and contain computations of the characteristic exponents for more general right-invariant diffusions on $\mathrm{GL}(n)$, including this as a special case.) From the leading exponent of $Y$ and the construction in Theorem~\ref{T: continuous time solution}, we may read off the critical value of the drift towards the mean opinion, above which~\eqref{E: Ito model} is stable: 

\begin{corollary}\label{C: critical gamma}
$\overline{Z}(t) \to 0$ and $Z(t)$ converges a.s.~as $t \to \infty$ if $\gamma > \frac{1}{2} - \frac{3}{2N}$. On the other hand, a.s.~$\overline{Z}(t)$ and~$Z(t)$ are both unbounded as~$t \to \infty$ if $\gamma < \frac{1}{2} - \frac{3}{2N}$. 
\end{corollary}
Although we will not give the details, we remark that an analogue of Theorem~\ref{t3} also holds for Model~B, with the main difference being that the exponential rate of convergence is simply $1/N$. Theorem~\ref{T: continuous time solution} and Corollary~\ref{C: critical gamma} are proven in Section~\ref{SecB}. The final result proven there is the following 
\begin{proposition}\label{P: covariance is Markov in model B} 
The process $\mathrm{Cov}^Z(t)$ is a Markov process in its own filtration.
\end{proposition}
For the reader who is primarily interested in Model~B, it may be helpful to know that Section~\ref{SecB} relies very little on Section~\ref{S: analysis of model A}, although it would be a good idea to read the discussion of the matrix~$V$, which precedes Lemma~\ref{lem: Ginibre ensembles} in~Section~\ref{Sec 2.1}, before turning to~Section~\ref{SecB}.

\subsection{Comparison of Model A and Model B}

The SDE~\eqref{E: Ito model} is intended to be a continuous time analogue of the discrete time model~\eqref{E: first definition of model A}. We leave it as a problem for the interested reader to determine whether~\eqref{E: Ito model} may be obtained as a scaling limit of Model~A as $\beta \to 1$ and $\alpha \to \infty$ in a suitable way. The computations of the conditional expectations and conditional variance of $\mathrm{Cov}^X(t)$ given $X(t-1)$ given in Proposition~\ref{prop: Covs are Markov} should be helpful in this task.  We should expect to lose one parameter in passing to the scaling limit, since the choice of the constant in the time rescaling will absorb one degree of freedom.

{
Let us compare the results that we prove about Model~A and Model~B:
\begin{itemize}
\item In both models the empirical covariance matrix of the opinions evolves as a Markov process on its own.
\item In both models we locate the phase transition between contractive and divergent asymptotic behaviour using Lyapunov exponents.
\item In both models the normalized opinions almost surely converge to a random one-dimensional subspace.
\item Model~A has two continuous parameters while Model~B has only one continuous parameter. 
\item 
The analysis of Model~A uses the theory of products of random matrices due to Furstenberg, Guivarc'h and Raugi. On the other hand, analysis of Model~B concerns diffusions on matrix groups and requires a good understanding of stochastic calculus, including the Stratonovich form.
\end{itemize}
}

\subsection{Possible extensions and open problems}\label{S: open problems}
\begin{itemize}
\item Does the statement of part (iii) of Theorem~\ref{t1} also hold for the case when $\lambda_1(\alpha,\beta,N)=0$?
\item  Find an analytic expression (in terms of standard special functions) for $\lambda_1(\alpha, \beta, N)$ when $N$ is even.
\item Study an asynchronous variant of Model~A. In this variant, at each time step just one individual is chosen, uniformly at random, to update their opinion. They update using the same sampling procedure that all the individuals use at each time step in Model~A. 
\item Generalize any of the results to weighted models, or to non-mean-field models, as described below. \end{itemize}

As we remarked above, our Models~A and~B are mean-field models, meaning that each individual is equally influenced by the other individuals. It may be also  interesting to study the following extension of Model~A, which is not mean-field, and simultaneously can be viewed as an extension of the DeGroot learning model. It is related to the model of Stern and Livan~\cite{SternLivan}, but it shares with our Models~A and~B the property of invariance under affine linear transformations of the opinion space. Take a weighted directed graph on $N$ vertices, with loops allowed. The vertices have positive weights $\beta_i$, $i=1, \dots, N$, and the directed edges have non-negative weights $w_{ij}$, $i, j  = 1,\dots,n$, where $w_{ij} = 0$ if the directed edge from $i$ to $j$ is not present in the directed graph, and for each $i$ we have $\sum_{j=1}^N w_{ij} =1$. As in Model~A, each individual has {\it an opinion vector} $X_i(t)$ on day $t$. Conditional on the matrix $X(t-1)$ whose rows are $X_1(t-1), \dots, X_N(t-1)$, the opinion vectors $X_i(t)$, $i=1, \dots, N$, are independent multivariate Gaussian random variables, whose conditional distribution is given by 
$$ 
\mathbb{E}(X_i(t) \,|\, X(t-1)) =  \beta_i X_i(t-1) + (1-\beta_i)\mu_i^X(t-1)\,,
$$
where 
$$
\mu_i^X(t-1) = 
\sum_{j=1}^N w_{ji}X_{j}(t-1) \,, 
$$
and
$$ 
\mathrm{Cov}(X_{ij}(t),X_{ik}(t) \,|\,X(t-1)) =  \alpha (1-\beta_i)^2 \,\sum_{j=1}^N w_{ji} w_{ki} (X_{ij} - \mu_{ij}^X(t-1))(X_{ik} - \mu_{ik}^X(t-1))\,;
$$
here $\mu_{ij}^X(t-1)$ stands for the $j^{th}$ coordinate of $\mu_i^X(t-1)$. This graph structure will allow for more complicated interactions and the emergence of clusters of similar opinions. It will also allow for heterogeneity of individuals and might include a geographical component or be based on a social network. On the other hand, this model would be much harder to tackle analytically.

Model~A is the special case of the above model with $w_{ij} = 1/N$ for all $i,j$ and $\beta_i = \beta$ for all $i$.  On the other hand the deterministic special case where $\alpha=0$ reduces to the well-known DeGroot learning model~\cite{DeGroot}.

We expect that the analysis of Model~A which we give in Section~\ref{S: analysis of model A} in terms of a product of a sequence i.i.d.~Gaussian random matrices may be extended to this more general model. However, the individual matrices will not in general be invariant in law under orthogonal conjugation, and therefore Newman's method for the exact computation of Lyapunov exponents, crucial in our argument, will not be applicable.

\section{Analysis of Model~A}\label{S: analysis of model A}
\subsection{Reduction to an i.i.d.~matrix product}\label{Sec 2.1}
 With some elementary linear algebra, we will reduce the question of convergence or divergence of the opinions to the computation of the maximal Lyapunov exponent of a product of i.i.d. random matrices constructed in a simple way from the real Ginibre ensemble.  The computation of all the Lyapunov exponents for a product of i.i.d.\ matrices drawn from the real Ginibre ensemble itself was performed by Newman \cite{Newman} in 1986, using Furstenberg's formula for Lyapunov exponents in terms of invariant measures. We apply Newman's method to deal with our matrix product. 

For any $k, \ell \ge 1$ let $I_k$ denote the $k$ by $k$ identity matrix, and let 
$$
\mathbf{1}_{k,\ell}=
\underbrace{
\left.
\begin{pmatrix}
    1 & 1& \dots &1 \\
    1 & 1& \dots &1 \\
    \vdots & \vdots& \ddots &\vdots \\
    1 & 1& \dots &1 
\end{pmatrix}
\right\}}
_{\text{$\ell$ columns}}
\text{$k$ rows}
$$
denote the $k$ by $\ell$ matrix with all entries equal to $1$. We will use $\mathbf{0}_{k,\ell}$ to denote the zero matrix with $k$ rows and $\ell$ columns, when it is helpful to emphasize the dimensions, but may simply write $0$ otherwise.

Let $W$ be any $N\times N$ square orthogonal matrix with final row  $\left(\frac{1}{\sqrt{N}}, \dots, \frac{1}{\sqrt{N}}\right)$. Thus the first $N-1$ rows of $W$ form an orthonormal basis for the vector subspace 
$$
S_N:=\{x \in \mathbb{R}^N \,:\, x_1 + \dots + x_N = 0\}\,.
$$ 
Let $V$ be the $N-1\times N$ submatrix of $W$ containing just the first $N-1$ rows. Then
$$ \mathbf{1}_{1,N} V^T = \mathbf{0}_{1,N}\,,$$
$$ 
V\,V^T = I_{N-1} 
$$
and 
$$
V^T V=
\frac1{N}\begin{pmatrix}
    N-1 & -1& \dots &-1 \\
    -1 & N-1& \dots &-1 \\
    \vdots & \vdots& \ddots &\vdots \\
    -1 & -1& \dots &N-1
\end{pmatrix}
=I_N-\frac1{N} \mathbf{1}_{N,N}\,.
$$
$V^TV$ is an $N\times N$ matrix of rank $N-1$, such that the sum of the elements in each row and each column equals~$0$, and such that for any column vector $w \in S_N$, $(V^T V) w = w$; since $ \mathbf{1}_{N,N} w$ is a zero vector. Consequently,  $V^T V$ is the orthogonal projector from $\mathbb{R}^N$ onto $S_N$. We will use the fact that $V^T V$ is idempotent: $$\left(V^T V\right)^2 = V^T(VV^T)V = V^T V\,.$$

\begin{lemma}\label{lem: Ginibre ensembles}
Let $N \ge 2$ and let $G^{(N)}$ be a random matrix drawn from the rank $N$ real Ginibre ensemble. (This means that $G^{(N)}$ is an $N$ by $N$ matrix with independent standard normal entries.) Let $G^{(N-1)}$ be a random matrix whose distribution is the rank $N-1$ real Ginibre ensemble. Let $F$ be an $N$ by $N-1$ matrix with independent standard Gaussian entries. Then 
$$
W G^{(N)} W^T \stackrel{d}{=} G^{(N)}\,,
$$
$$ 
W G^{(N)} V^T \stackrel{d}{=} F\,,\; \text{ and}
$$
$$
VG^{(N)} V^T \stackrel{d}{=} G^{(N-1)}\,.
$$
\end{lemma}
\begin{proof}
It is well known that the real Ginibre ensembles are invariant under conjugation by orthogonal matrices. This is easy to check using the fact that the law of a mean zero Gaussian vector is determined by the covariances of its entries. Indeed, for any orthogonal matrix $U$ that is either deterministic or random but independent of $G^{(N)}$ we have
\begin{align*}
\mathbb{E}((U G^{(N)} U^T)_{i,j}&(U G^{(N)} U^T)_{i',j'})  \;=\; \mathbb{E} \sum_{k,l,k',l' = 1}^{N} U_{i,k} U_{i',k'} G^{(N)}_{k,l} G^{(N)}_{k',l'} U_{j,l} U_{j',l'}\\ 
& = \sum_{k=1}^N U_{i,k}U_{i'k} \sum_{l=1}^N U_{j,l}U_{j',l} 
 \;=\; (UU^T)_{i,i'}\,(UU^T)_{j,j'} \;=\; \delta_{i,i'}\delta_{j,j'}\,.
\end{align*}
Taking $U = W$, this implies 
\begin{equation}\label{E: Ginibre invariance} W G^{(N)} W^T  = \tilde{G}\,,\end{equation}
where $\tilde{G} \stackrel{d}{=} G^{(N)}$.  Removing the final columns on both sides of~\eqref{E: Ginibre invariance} we obtain
\begin{equation}\label{E: truncated Ginibre invariance} W G^{(N)} V^T  = F \end{equation}
and removing the final rows on both sides of~\eqref{E: truncated Ginibre invariance} we obtain
{$$ V G^{(N)} V^T = G^{(N-1)}\,.$$}
\end{proof}

Let $X(0)$ be the $N\times d$ matrix whose rows are the vectors $X_1(0), \dots, X_N(0)$. Let $\bar{X}(0) = V^TV X(0)$.  Thus $\bar{X}(0)$ is the $N$ by $d$ matrix whose rows are $X_1(0) - \mu(0)$, $X_2(0) - \mu(0)$, $\dots$, $X_N(0) - \mu(0)$. The sum of the rows of $\bar{X}(0)$ is $0 \in \mathbb{R}^d$; in other words each column of $\bar{X}(0)$ belongs to $S_N$. We can express this in matrix form as
$ \mathbf{1}_{1,N} \bar{X}(0) = 0$.

\begin{lemma}\label{LS: X and Xdash} Let $G^{(N)}(t)$, $t=1,2,\dots$, be i.i.d. random matrices drawn from the rank $N$ real Ginibre ensemble, and for each $t \ge 1$ define
$$
S(t) :=  \frac1N{\bf 1}_{N,N} +
\left[\beta I_N
+\sqrt{\frac{\rho}N} G^{(N)}(t)  \right]V^T V\,, 
$$
where $\rho$ is given by~\eqref{defrho} and $\alpha,\beta$ are introduced around~\eqref{E: first definition of model A}.
Then the sequence $\left(X(t)\right)_{t \ge 0}$  defined inductively from the starting $N$ by $d$ matrix $X(0)$ by
\begin{align} \label{XX1}
X(t)= S(t) X(t-1)
\end{align}
has the distribution described in Model~A, meaning that we can take the vectors $X_1(t), \dots, X_N(t)$ in Model~A to be the~$N$ rows of $X(t)$, for each $t \ge 0$. 
\end{lemma}
\begin{proof}  
Let $t \ge 1$. For simplicity of notation, we will write $X = X(t-1)$,  $\mu = \frac{1}{N}\mathbf{1}_{1,N} X$, $\bar{X} = V^T V X$, $G = G^{(N)}(t)$,  and $X_{new}:=X(t)$, $\mu_{new} = \frac{1}{N}\mathbf{1}_{1,N} X_{new}$, $\bar X_{new} = V^T V X_{new}$. Then $\mu$ equals each row of $\frac1N{\bf 1}_{N,N} X$, and thus
$$
\E \left[X_{new}\mid X\right]=\left[\beta I_N+\frac{1-\beta}N{\bf 1}_{N,N} \right]X.
$$
For each $k \in \{1, \dots, N\}$, let $G_k$ denote the $k^{th}$ row of $G$ and let $\eta^{(k)} = G_k \bar X$. 
$$
\Cov\left(\eta^{(k)}_i,\eta^{(k)}_j\right)=
\Cov\left(\sum_{m=1}^N  G_{k,m} \bar  X_i^{(m)} ,\sum_{l=1}^N  G_{k,l} \bar  X_j^{(l)}\right)=
\sum_{m=1}^N \bar  X_{m,i}\bar  X_{m,j} =N\cdot \Cov^X_{i,j}\,.
$$
Also, conditional on $X(t-1) = X$, the row vectors $\eta^{(1)}, \dots, \eta^{(N)}$ are independent, because the rows of $G$ are independent. So if
$$
X_{k,new}=\E \left[X_{k,new}\mid X\right]+\sqrt{\frac{\rho}N}\eta^{(k)}.
$$
then $X_{k,new}$ indeed has the distribution required by Model~A. The matrix whose rows are $\eta^{(k)}$, $k=1,2,\dots,N$, equals $G\bar X$. Hence, by combining the above, we get
\begin{align*} 
X_{new}&=\left[\beta I_N+\frac{1-\beta}N{\bf 1}_{N,N} \right]X
+\sqrt{\frac{\rho}N} G \bar X
=\left[\beta I_N+\frac{1-\beta}N{\bf 1}_{N,N} \right]X
+\sqrt{\frac{\rho}N} G V^T V X
\\ 
&=\left[\frac1N{\bf 1}_{N,N} +
\beta V^T V\right]X
+\sqrt{\frac{\rho}N} G V^T V X
\\ &
=\frac1N{\bf 1}_{N,N} X+
\left[\beta I_N
+\sqrt{\frac{\rho}N} G  \right]V^T V X
= S(t) X.
\end{align*}
\end{proof}
From now on we will work with the sequence $\left(X(t)\right)_{t \ge 0}$ that is generated from the initial value $X(0)$ and the i.i.d. random matrix sequence $\left(G^{(N)}(t)\right)_{t \ge 1}$, as described in Lemma~\ref{LS: X and Xdash}. We define a filtration by letting~$\mathcal{F}_t$ be the $\sigma$-algebra generated by $X(0)$ and $\left(G^{(N)}(i)\right)_{1 \le i \le t}$, for each $t \ge 0$. We write $\mu(t) = \frac{1}{N} \mathbf{1}_{N,N} X(t)$ for the empirical mean vector at time $t$, which was denoted $\mu^X\!(t)$ in the description of Model~A. We write $\bar{X}(t) = V^TV X(t)$ for the $N$ by $d$ matrix whose rows are $X_1(t) - \mu(t)$, $X_2(t) - \mu(t)$, $\dots$, $X_N(t) - \mu(t)$. For each $t \ge 0$, the sum of the rows of $\bar{X}(t)$ is $0 \in \mathbb{R}^d$; in other words, the columns of $\bar{X}(t)$ belong to $S_N$. We can express this as $ \mathbf{1}_{1,N} \bar{X}(t) = 0$. 
\begin{corollary}
We have
\begin{align*} 
\bar{X}(t)&=
V^T V \left[\beta I_N
+\sqrt{\frac{\rho}N} G^{(N)}(t)  \right]\,\bar{X}(t-1).
\end{align*}
\end{corollary}
\begin{proof}
Since $V^T V{\bf 1}_{N,N}=0$, from~\eqref{XX1} we have, for each $t \ge 1$,
\begin{align*}
\bar{X}(t) &= V^T V X (t) =
V^T V \left[\beta I_N
+\sqrt{\frac{\rho}N} G^{(N)}(t)  \right] V^T V X(t-1) \\
& =V^T V \left[\beta I_N
+\sqrt{\frac{\rho}N} G^{(N)}(t)  \right] \bar X(t-1)\,.
\end{align*}
\end{proof}

\begin{lemma}\label{lem: L(t)}
For each $t \ge 1$, let $L(t) := W S(t) W^T$. Then the entries of $L(t)$ are independent and may be represented as
$$ L_{i,j}(t) =  \begin{cases}
 \sqrt{\frac{\rho}N} F_{i,j}(t) + \beta \delta_{i,j} & \text{if $1 \le j \le N-1$,}\\
 \delta_{i,N} &\text{if $j = N$.} \end{cases}$$
 Here $\delta_{i,j}$ is $1$ if $i=j$ and $0$ otherwise, and $\left(F(t)\right)_{t \ge 0}$ is an i.i.d.~sequence of $N$ by $N-1$ matrices with independent standard Gaussian entries.
\end{lemma}
\begin{proof}
We have
\begin{equation}\label{eq: WSWT} 
WS(t)W^T   =  W\left[\frac{1}{N}\mathbf{1}_{N,N} + \beta V^T V\right]W^T + \sqrt{\frac{\rho}{N}} W G^{(N)}(t) V^T V W^T\,.
\end{equation}
For the first term on the RHS of~\eqref{eq: WSWT}, we have
$$
W\left[\frac{1}{N}\mathbf{1}_{N,N} + \beta V^T V\right]W^T = W\left[\beta I_N + \frac{1-\beta}{N} \mathbf{1}_{N,N} \right]W^T  = \mathrm{Diag}(\beta, \dots, \beta, 1)\,.
$$
Here $\mathrm{Diag}(\beta, \dots, \beta,1)$ denotes the $N$ by $N$ diagonal matrix whose diagonal entries are all equal to $\beta$ except for the last one, which is equal to $1$. For the second term on the RHS of~\eqref{eq: WSWT}, by the second equation of Lemma~\ref{lem: Ginibre ensembles} we may equate $W G^{(N)}(t)V^T$ with $F(t)$. Now, $VW^T$ is $I_N$ with its final row removed, or equivalently $I_{N-1}$ with an extra column of zeros added on its right.  Hence $WG^{(N)}(t)V^T VW^T$ is equal to $F(t)$ with an extra final column of zeros. 
\end{proof}

Lemma~\ref{lem: L(t)} gives us an alternative representation of the sequence $\left(X(t)\right)_{t \ge 0}$ in terms of the initial value $X(0)$ and the i.i.d.~sequence $\left(L(t)\right)_{t \ge 0}$: for each $t \ge 1$ we have
$$ 
X(t) = W^T L(t) W X(t-1)\,,
$$
or
$$ 
WX(t) = L(t) W X(t-1)\,.
$$
The matrix consisting of the first $N-1$ rows of $WX(t)$ is $VX(t)$. Note that each $L(t)$ is a block-lower-triangular matrix, with two blocks of size $N-1$ and $1$, so $\left(VX(t)\right)_{t \ge 0}$ is a Markov chain on its own: 
if we denote by $M(t)$  the top-left $(N-1)\times(N-1)$ submatrix of $L(t)$, then 
$$ 
\left(VX(t)\right) = M(t) \left(V X(t-1)\right)\,,
$$ 
For every $t \ge 0$ we have 
$$ 
VX(t) = (VV^T)VX(t) = V(V^T V) X(t) = V \bar X(t)\,.
$$ 
{Note that $V \bar X(t)$ determines $\bar X(t)$ and $\bar X(t)$ determines $V \bar X(t)$. This is because left-multiplication by $V$ acts as an injective linear map on the subspace $S_N$.} Hence $\left(\bar X(t)\right)_{t \ge 0}$  is also a Markov chain on its own. Indeed, it satisfies
\begin{equation}\label{E: Xbar recursion} 
\bar{X}(t) = V^T VX(t) = V^TM(t)VX(t-1) =  V^T M(t) V \bar{X}(t-1)\,.
\end{equation}
Since $V^TV = I_N - \frac{1}{N}\mathbf{1}_{N,N} = I_N - \frac{1}{N}\mathbf{1}_{N,1}\mathbf{1}_{1,N}$, for each $t \ge 0$ we have
\begin{equation}\label{E: X(t) as a sum} 
X(t) = \bar{X}\!(t) + \mathbf{1}_{N,1}\mu(t)\,.
\end{equation}
We may check from the definition of $S(t)$ that
$$ 
S(t) \mathbf{1}_{N,1} = \mathbf{1}_{N,1}\,. 
$$
(By Lemma~\ref{lem: L(t)} this is equivalent to the statement that the final column of $L(t)$ is $(0, \dots, 0,1)^T$.)
Hence
$$
X(t) = S(t) X(t-1) = S(t) \bar{X}(t-1)  + S(t)\mathbf{1}_{N,1} \mu(t-1) = S(t) \bar{X}(t-1)  + \mathbf{1}_{N,1} \mu(t-1)\,.
$$
An induction now shows that for each $t \ge 0$ we have
$$ 
X(t) = S(t) \dots S(1) \bar{X}(0) + \mathbf{1}_{N,1} \mu(0)\,.
$$
Recall $M(t)$ defined by~\eqref{E: M definition}
%We may alternatively  write $$M(t) = \sqrt{\frac{\rho}{N}} G^{(N-1)}\!(t) + \beta I_{N-1}\,,$$where $\left(G^{(N-1)}\!(t)\right)_{t \ge 0}$ is an i.i.d.~sequence drawn from the rank $(N-1)$ real Ginibre ensemble. 
and note that $G^{(N-1)}(t)$ is obtained from $F(t)$ by removing its final row. Define now
$$
\bar{S}(t) = V^T M(t) V\,.
$$  Then~\eqref{E: Xbar recursion} becomes 
\begin{equation}\label{eq:recur}
\bar{X}(t) = \bar{S}(t) \bar{X}(t-1) = V^{T}\left(
\sqrt{\frac{\rho}N}G^{(N-1)}\!(t)
+ \beta I_{N-1}\right)V \bar{X}(t-1)\,.  
\end{equation} 
or, using the fact that $V^T V \bar X(t-1) = (V^T V)^2 X(t-1) =  (V^T V) X(t-1) = \bar X(t-1)$,  
\begin{equation}\label{E: X tilde} 
\bar{X}(t) = 
\sqrt{\frac{\rho}N}
V^T G^{(N-1)}\!(t)\,V\,\bar{X}(t-1) \,+\, \beta \bar{X}(t-1)\,.\end{equation}
Another calculation shows that
\begin{equation}\label{eq: mu increment}
\mu(t) - \mu(t-1) = \frac{\sqrt{\rho}}{N} F_N(t)\, V \,\bar{X}(t-1) \,,
\end{equation}
where $F_N(t)$ is the final row of $F(t)$. We see that for every $t \ge 1$, the random variables $\mu(t) - \mu(t-1)$, $\bar X(t)$, and $\left(X(s)\right)_{0 \le s < t-1}$ are conditionally independent given $\bar X(t-1)$. Since $\bar{X}(t-1)$ and $\mu(t-1)$ are both functions of $X(t-1)$, it follows that the two terms on the right-hand side of~\eqref{E: X(t) as a sum} are conditionally independent given~$X(t-1)$.

From~\eqref{eq: mu increment} we may check that the conditional distribution of $\mu(t) - \mu(t-1)$ given $\bar{X}(t-1)$ is a d-dimensional Gaussian with mean $0$ and covariance $\frac{\rho}{N^2} \bar X(t-1)^T \bar X(t-1) = \frac{\rho}{N} \Cov^X(t-1)$.

Iterating~\eqref{eq:recur} and using $VV^T = I_N$, for each $t \ge 0$ we obtain
\begin{equation}\label{eq: Xbar in terms of Lbar product} 
\bar{X}(t)  =  \bar{S}(t) \dots \bar{S}(1) \bar{X}(0) = V^T M(t) \dots M(1) V \bar{X}(0)\,.
\end{equation}
(In the case $t=0$, the empty product of matrices is to be interpreted as the identity matrix $I_N$.)
Using equation~\eqref{eq: mu increment} to evaluate the telescoping sum $\mu(t) - \mu(0) = \sum_{r = 1}^t \mu(r) - \mu(r-1)$, it follows that
\begin{equation}\label{eq: complicated} 
X(t) = \bar{S}(t) \dots \bar{S}(1)\bar{X}(0) \;+\; \mathbf{1}_{N,1}\mu(0) \;+\; \mathbf{1}_{N,1}\sum_{r=1}^t \sqrt{\frac{\rho}{N}} F_N(r) V \bar{S}(r-1) \dots \bar{S}(1)\bar{X}(0)\,.
\end{equation}
Since the random variables $\bar{S}(1), \bar{S}(2), \dots $ and $F_N(1), F_N(2), \dots$ are independent, the formulation~\eqref{eq: complicated} makes the conditional independence structure explicit. We have shown the following statement, which will be useful later:

\begin{lemma}\label{lem: conditional independence}
The sequence $\left(\overline{X}(t)\right)_{t \ge 0}$ is a Markov chain on its own, and conditional on $\left(\overline{X}(t)\right)_{t \ge 0}$ the sequence of mean increments $\left(\mu(t) - \mu(t-1)\right)_{t \ge 1}$ is a sequence of independent Gaussians with covariance matrices given by
$$ 
\Cov\left(\mu(t) - \mu(t-1) \,|\, \left( \overline{X}(t)\right)_{t \ge 0}\right)\;=\; \frac{\rho}{N} \Cov^X(t-1)\,.
$$
\end{lemma}

\begin{lemma}\label{lem: condition for mean opinion to converge} 
The sequence of means $\mu(1), \mu(2), \dots $ converges almost surely if and only if the series $\sum_{t=0}^\infty \Cov^X(t)$ converges (element-wise) almost surely. On the event that both converge, all the individual opinions converge to $\lim_{t \to \infty} \mu(t)$.
\end{lemma}
\begin{proof}
By Kolmogorov's one series theorem, for each $i = 1, \dots, d$, the random series $\sum_{t=1}^\infty [\mu_i(t) - \mu_i(t-1)]$  converges almost surely conditional on $\left(\bar{X}(t)\right)_{t \ge 0}$ if and only if 
\begin{equation}\label{E: conditional variance series} 
\sum_{t=1}^\infty \mathrm{Var}\left(\mu_i(t) - \mu_i(t-1) \mid \left(\bar{X}(s)\right)_{s \ge 0} \right) < \infty\,.
\end{equation} 
Condition~\eqref{E: conditional variance series} holds if and only if
\begin{equation}\label{E: ii} \sum_{t=0}^\infty 
\Cov^X_{ii}(t) < \infty.
\end{equation}
If~\eqref{E: ii} holds for each $i=1, \dots, d$, then since for each $i \neq j$ we have  $2\left|\Cov^X_{i,j}(t)\right| \le \Cov^X_{i,i}(t) + \Cov^X_{j,j}(t)$, we must have $\sum_{t=0}^\infty \Cov^X(t) < \infty$.

For the second statement, note that if $\sum \Cov^X(t)$ converges then $\Cov^X(t) \to 0$ as $t \to \infty$, and since 
$$
\frac1N\, \sum_{j=1}^N \| X_j(t) - \mu(t)\|^2 = \mathrm{Tr} (\Cov^X(t))$$
we find that 
$$
\max_{j=1,\dots,N} \|X_j(t) - \mu(t) \| \to 0\quad \text{
as $t \to \infty$.}
$$ 
Thus if the sequence $\mu(t)$  converges to some limit $\mu_\infty$, then $X_j(t)$ also converges to $\mu_\infty$ for each $j\in\{1,\dots,N\}$. 
\end{proof}

We now wish to discover for which choices of the parameters $\alpha, \beta, N$ the series $\sum_{t=0}^\infty \Cov^X(t)$ converges almost surely.
Define
$$
P(t) := M(t) M(t-1) \dots M(1)\,.
$$  
We will prove that the series $\sum_{t=0}^\infty \Cov^X(t)$ a.s.~converges if and only if the leading Lyapunov exponent for the i.i.d.~random matrix product sequence $P(\cdot)$ is strictly negative, and we will identify necessary and sufficient conditions on the parameters $(N,\alpha, \beta)$  for this to occur by computing this Lyapunov exponent explicitly. By the argument shown in the proof of Lemma~\ref{lem: Ginibre ensembles}, for each $t \ge 1$ the law of $M(t)$ is invariant under conjugation by an arbitrary element $U$ of the orthogonal group $O(N-1)$, where $U$ may be random as long as it is independent of $G(t)$. This fact will be the key to evaluating the Lyapunov exponents, and it is used in the proof of Proposition~\ref{prop: Covs are Markov}, which we now give.
\begin{proof}[Proof of Proposition~\ref{prop: Covs are Markov}]
Let $t \ge 0$. We have $\Cov^X(t) = \frac{1}{N}\bar{X}(t)^T \bar{X}(t)$.  By equation~\eqref{eq: Xbar in terms of Lbar product}, $\bar{X}(t)^T\,\bar{X}(t)$ has the same distribution as $\bar{X}(0)^T V^T P(t)^T V V^T P(t) V \bar{X}(0)$, which simplifies to 
$$
\bar{X}(t)^T\,\bar{X}(t)\;\stackrel{(d)}{=}\;\bar{X}(0)^T V^T P(t)^T P(t) V \bar{X}(0)\,.
$$ 
Note that the vector $V\bar{X}(0)$ can be any random $(N-1)$ by $d$ matrix that is independent of the sequence $\left(P(t)\right)_{t \ge 1}$.   Write $A = N\Cov^X(t)$. Then  $A$ is a non-negative definite symmetric matrix; in fact $A=K^T K$ where $K = P(t) V \bar{X}(0)$. So
$$
\bar{X}(0)^T V^T P(t+1)^T P(t+1) V \bar{X}(0) = K^T M(t+1)^T M(t+1) K\,.
$$
We claim that the conditional law of $K^T M(t+1)^T M(t+1) K$ given $K$ is a function of $A$. Indeed, suppose $K'$ is some other real $(N-1)\times d$ matrix such that $K'^T K' = A$. Then there exist partial polar decompositions of $K$ and $K'$ of the form $K = Q D O$, $K' = Q' D O$, where $O$ is an orthogonal $d\times d$ square matrix, $D$ is an $(N-1)\times d$ matrix whose $(i,j)$ entry is $0$ if $i \neq j$ and non-negative if $i=j$, and $Q$ and $Q'$ are orthogonal $(N-1)\times (N-1)$ square matrices which are random but may be taken to be independent of $M(t+1)$.  Then note that 
$$
Q^T M(t+1)^T M(t+1) Q = Q^T M(t+1)^T Q \, Q^T M(t+1) Q = R^T R\,,
$$ 
where $R:=Q^T M(t+1) Q$. But $R$ has the same law as $Q'^T M(t+1) Q'$, because the law of $M(t+1)$ is invariant under conjugation by the element $Q^{-1}Q' \in O(N-1)$, which is independent of $M(t+1)$.  So 
$$ 
K^T M(t+1)^T M(t+1) K \stackrel{(d)}{=}  K'^T M(t+1)^T M(t+1) K' \,.
$$
The formulae for the conditional expectation and conditional variance of~$\mathrm{Cov}^X(t)$ given~$\mathcal{F}^X_{t-1}$ are obtained by a long but elementary calculation which we omit.
\end{proof}

\begin{lemma}\label{L: empirical variances of single coordinates are random walks}
For each $j = 1, \dots, d$, the sequence $\left(\log \mathrm{Cov}^X_{j,j}(t)\right)_{t \ge 0}$ is a random walk with i.i.d.~increments. The increments are distributed as the log of a non-central $\chi^2$ random variable. 
\end{lemma} 
\begin{proof} We noted in the introduction that Model~A is invariant under affine changes of coordinate. This extends to affine maps to opinion spaces of different dimensions. If $A$ is any $N\times d$ real matrix and $U$ is any $d\times d'$ real matrix, then the law of $(X(t)U)_{t \ge 0}$ given $X(0)  = A$ is equal to the law of $\left(X(t)\right)_{t \ge 0}$ given $X(0) = AU$. Likewise, if $C$ is any non-negative definite symmetric $d\times d$ matrix then the law of  $\left(U^T\Cov^X(t)U\right)_{t \ge 0}$ given $\Cov^X(0) = C$ is equal to the law of $\left(\Cov^X(t)\right)_{t \ge 0}$ given $\Cov^X(0) = U^T C U$. In the special case $d=1$, the covariance matrix $\Cov^X(t)$ is simply the empirical variance $\mathrm{Var}^X(t)$, and the proof of Proposition~\ref{prop: Covs are Markov} shows that $\log(\mathrm{Var}^X(t))$ is a random walk with i.i.d~increments which have the distribution of $\log \|M_1(1)\|^2$, where $M_1(1)$ is the first row of $M(1)$. The quantity $\|M_1(1)\|^2$ has a non-central $\chi^2$ distribution. For each $j=1,2,\dots,d$, applying affine invariance with $U$ being the $d\times 1$ matrix given by $U_{i1} = \delta_{ij}$, we find that the matrix entry $\Cov^X_{j,j}(t)$ on the diagonal also follows a random walk with the same description.
\end{proof}

\subsection{Analysis of the i.i.d.~matrix products}
In this section, we analyze the sequence $(P(t))_{t \ge 1}$ by applying some of the standard theory of Lyapunov exponents of i.i.d.~random matrix products. The reader may find all of the necessary theory in the monograph of Bougerol, which is part A of the two-part book~\cite{BougerolLacroix}. For completeness, we include a review of the basic definitions and results, which readers familiar with the topic may skip.

\subsubsection{Brief review of Lyapunov exponents of i.i.d.~matrix products}\label{Lyapexp}
The sequence $\left(P(t)\right)_{t \ge 1}$ is an example of a sequence of partial products of an i.i.d.~sequence of random matrices. A comprehensive theory of the behaviour of such sequences was developed from the 1960s to 1980s by Furstenberg, Kesten, Kifer, Guivarc'h, Raugi and Le Page, and some specific results relevant to the matrix products considered in the present paper were published by Newman and Cohen~\cite{CohenNewman}.  

Let $\mu$ be a Borel probability measure on the group $GL_k(\mathbb{R})$ of invertible $k$ by $k$ matrices and let $Y_1, Y_2, \dots$ be an i.i.d.~sequence of matrices distributed according to $\mu$. 
For $Y \in \mathrm{GL}_k(\mathbb{R})$ let $\|Y\|$ denote the operator norm of $Y$. Assume that $\mathbb{E}(\log^{+}\|Y_1\|) < \infty$. The leading (or upper, or maximal) Lyapunov exponent $\lambda_1$ associated to $\mu$ is defined by 
$$
\lambda_1 = \lim_{n \to \infty} \frac{1}{n} \mathbb{E}( \log \|Y_n \dots Y_1\|)\,.
$$
This limit exists in $\mathbb{R} \cup \{-\infty\}$ since the sequence $\mathbb{E}(\log \|Y_n \dots Y_1\|)$ is subadditive.  Furstenberg and Kesten showed that in fact
$$
\gamma_1 = \lim_{n \to \infty} \frac{1}{n}\log\|Y_n \dots Y_1\|\quad\text{a.s.}
$$
This may be shown quite easily as a consequence of Kingman's subadditive ergodic theorem since the random sequence $\log \|Y_n \dots Y_1\|$ is also subadditive.

Denote by $T_\mu$ the smallest closed sub-semigroup of $GL_k(\mathbb{R})$ containing the support of $\mu$. $T_\mu$ is called \emph{contracting} if there exists a sequence $\left(A_n\right)_{n \ge 1}$ of matrices in $T_\mu$ such that $\|A_n\|^{-1} A_n$ converges to a rank one matrix. $T_\mu$ is called \emph{strongly irreducible} if there is no non-empty finite collection of proper linear subspaces of~$\mathbb{R}^{N-1}$ that $Y_1$ a.s. preserves set-wise. In particular if $T_\mu = GL_k(\mathbb{R})$ then $T_\mu$ is both contracting and strongly irreducible. Assume now that $T_\mu$ is contracting and strongly irreducible, and in addition that
$$
\mathbb{E}(\log(\max(\|Y_1\|, \| Y_1^{-1}\|))) < \infty\,.
$$ 

The Lyapunov exponents $\lambda_1, \dots, \lambda_{k}$ are defined for each $i = 1, \dots, k$ by
$$\lambda_1 + \dots + \lambda_i = \max \lim_{t \to \infty} \frac{1}{t}\log\|P(t)v_1 \wedge \dots \wedge P(t)v_i\|\,,$$ where the maximum is over all choices of $i$ vectors $v_1, \dots, v_i \in \mathbb{R}^{i}$, and $\|w_1 \wedge \dots \wedge w_i\|$ denotes the $i$-dimensional volume of the parallelepiped spanned by vectors $w_1, \dots, w_i$. The limits all exist 
 and are finite, and the maximum is also finite.  Moreover, $\lambda_1 \ge \lambda_2 \ge \dots \ge \lambda_k$.
  
In particular, the maximal Lyapunov exponent is
$$
\lambda_1 =  \max_v \lim_{t \to \infty} \frac{1}{t}\log \|P(t)v\|\,.
$$ 
The maximum here is in fact achieved by all $v$ outside a certain (random) codimension-one subspace of $\mathbb{R}^{k}$, and for each $v \in \mathbb{R}^{k} \setminus \{0\}$ we have
$$\lim_{t \to \infty} \frac{1}{t} \log \|P(t)v\|= \lambda_1 \;\text{ a.s.}  $$

The sum of the first $k$ Lyapunov exponents is the maximal Lyapunov exponent of the induced action of the matrices $P(t)$ on the exterior power $\bigwedge^k \mathbb{R}^{k}$. In particular the sum of all $k$ Lyapunov exponents is 
$$ 
\lambda_1 + \dots + \lambda_{k} = \lim_{t \to \infty} \frac{1}{t}\log \det P(t), \quad \text{a.s.}
$$ 

Letting the eigenvalues of $P(t)^TP(t)$ be $e_1(t) \ge e_2(t) \ge \dots \ge e_{k}(t)$ (repeated according to multiplicity), we have 
$$
\lambda_k = \lim_{t \to \infty} \frac{1}{2t} \log  e_k(t)\,.
$$

The Lyapunov exponents are all the possible values of $ \lim_{t \to \infty} \frac{1}{t}\log \|P(t)v\|$ as $v$ ranges over $\mathbb{R}^{k} \setminus \{0\}$.

{Any matrix $Y$ in $\mathrm{GL}_k(\mathbb{R})$ acts on the projective space $P(\mathbb{R}^k) = \{ \langle v \rangle\,:\, v \in \mathbb{R}^k\setminus\{0\}\}$ by sending the line $\langle v \rangle = \{sv: s \in \mathbb{R}\}$ to the line $\langle Yv\rangle$.} Guivarc'h and Raugi \cite{GR} proved that when $T_\mu$ is contracting and strongly irreducible, we have $\lambda_1 > \lambda_2$. Moreover, there is exponential contraction on projective space at rate $\lambda_1-\lambda_2$. We measure distance on projective space using the metric
$$ 
\delta(\overline{v},\overline{w})  = \sin(\theta(v,w))
$$ 
where $\theta(v,w)$ is the angle between any two non-zero vectors $v$ and $w$ in $\mathbb{R}^{k}$ and $\overline{v}$ and $\overline{w}$ are their classes in~$P\left(\mathbb{R}^{k}\right)$.  Then for any two deterministic non-zero vectors $v, w \in \mathbb{R}^{k}\setminus \{0\}$, almost surely
\begin{equation}\label{E: exponential contraction} \limsup_{t \to \infty} \frac{1}{t} \log \delta\left(\overline{P(t)v},\overline{P(t)w}\right) \,\le\, -(\lambda_1 - \lambda_2)\,<\, 0\,,
\end{equation} 
where we use the convention that $\log 0 = -\infty$. There may also exist random vectors $v$ and $w$, depending on the whole sequence $\left(Y_t\right)_{t \ge 1}$, for which~\eqref{E: exponential contraction} fails.

There is a unique invariant probability measure $\pi$ on projective space $P(\mathbb{R}^{k})$. Invariance means that if $v$ is a random non-zero vector in $\mathbb{R}^{k}$ such that the class of $v$ in $P\left(\mathbb{R}^{k}\right)$ is distributed according to $\pi$, and $M$ is a matrix distributed like $M(1)$, independent of $v$, then the class of $Mv$ in $P\left(\mathbb{R}^{k}\right)$ is also distributed according to $\pi$. The support of $\pi$ is all of $P\left(\mathbb{R}^{k}\right)$. Finally, $\lambda_1$ may be expressed as an integral with respect to $\pi$: 
$$
\lambda_1 = \int \mathbb{E}\left(\log \frac{\|Mv\|}{\|v\|} \right) d\pi(v) \,.
$$

\subsubsection{\texorpdfstring{The maximal Lyapunov exponent of $P(t) = M(t) \dots M(1)$ for $t \ge 0$}{Lg}}\label{SS: max exponent}
Following Newman \cite{Newman}, note that for our random matrix $M$, for any orthogonal matrix $Q \in O(N-1)$, the law of $Q^{-1} M Q = Q^T M Q$ coincides with the law of $M$. (One just has to check that the entries of $Q^T M Q$ are pairwise uncorrelated Gaussians with the correct means and variances, a simple calculation.) It follows that the invariant measure $\pi$ is also invariant with respect to the group $O(N-1)$. There is a unique such measure on $P\left(\mathbb{R}^{N-1}\right)$. It also follows that the distribution of $\frac{\|Mv\|}{\|v\|}$ does not depend on $v$, so 
$$
\lambda_1 = \mathbb{E}\left(\log \|M_1\|\right)\,,    
$$
where $M_1$ is the first row of $M$. Recalling the notation $\rho = \alpha(1-\beta)^2$, we have
\begin{equation}\label{E: maximal exponent}\lambda_1 = \lambda_1(\alpha,\beta,N) := \frac{1}{2}\mathbb{E}\left[\log\left(\left(\beta + \frac{\sqrt{\rho}}{\sqrt{N}}X_1\right)^2 + \sum_{i=2}^{N-1} \frac{\rho}{N} X_i^2 \right) \right]\,,
\end{equation} 
where $X_1, \dots, X_{N-1}$ are independent standard Gaussians. Notice that when $N=2$ the second term inside the logarithm is an empty sum. Let $\chi^2_{\nu}(\kappa)$ denote the non-central $\chi^2$ distribution with $\nu$ degrees of freedom and non-centrality parameter $\kappa$. Then equation~\eqref{E: maximal exponent} can be interpreted  as
\begin{align}\label{eq:lambda1}
\lambda_1 = \frac{1}{2}\left(\log \frac{\rho}{N} + \mathbb{E}(\log Y)\right)\, 
\end{align}
where $Y\sim \chi^2_{N-1}\!\left(\frac{N\beta^2}{\rho}\right)$. Next, for $Y_{\nu,\kappa} \sim \chi^2_\nu(\kappa)$, we have
\begin{equation}\label{E: expectation of log non-central chi-squared}
\mathbb{E}(\log Y_{\nu,\kappa}) = \log 2 +   e^{-\kappa/2} \sum_{j=0}^\infty \frac{1}{j!}
\left(\frac{\kappa}{2}\right)^j
\psi\left(j+\frac{\nu}{2}\right)\,, \end{equation}
where $\psi(x)=\frac{\Gamma'(x)}{\Gamma(x)}$ is the digamma function. This follows from the statement that each non-central chi-squared random variable is a Poisson-weighted mixture of standard chi-squared random variables. Precisely, if $J \sim \mathrm{Poi}(\kappa/2)$ and, conditional on $J=j$, $Y$ is a $\chi^2$ r.v. with $\nu + 2j$ degrees of freedom, then $Y \sim \chi^2_{\nu}(\kappa)$. See also e.g.\ the proof of Lemma~10.1 in~\cite{Capacity}.

Putting this together, we have 
\begin{equation} \label{E: max exponent as series} 
\begin{split}
\lambda_1 &= \frac{1}{2}\left[ \log \frac{2\rho}{N} +  \exp\left(-\frac{N\beta^2}{2\rho}\right)\,\sum_{j=0}^{\infty}\frac{\left(\frac{N\beta^2}{2\rho} \right)^j}{j!} \psi\left(j + \frac{N-1}{2} \right)\right]
\\ & =\frac{1}{2}\left[ \log \frac{2\rho}{N} +  \phi_{(N-1)/2}\left(\frac{N\beta^2}{2\rho}\right)\right]
=\log \beta+
\frac{1}{2}\left[  \phi_{(N-1)/2}\left(z\right)-\log z\right]
\,,
\end{split}
\end{equation}
(the last equality holds only when $\beta>0$)
where $z=N\beta^2/(2\rho)$ and
\begin{align*}
\phi_m(x)=e^{-x}\sum_{j=0}^\infty\frac{x^j}{j!}\psi(j+m),\qquad \phi_m(0)=\psi(m)\,.
\end{align*}
This expression is the one stated in Theorem \ref{t1}; the proof of the second part of equation~\eqref{def of phi} is deferred to Lemma~\ref{lem:phi}.

By conditioning on $W_k$ and using the tower law for expectation we may rewrite the expression for~$\lambda_k$ as
$$
\lambda_k(\alpha, \beta,N) = \frac{1}{2}\left(\log \frac{\rho}{N} + \mathbb{E}(\mathbb{E}(\log Y_k | W_k))\right)\, 
$$
where conditional on $W_k$ we have $Y_k\sim \chi^2_{N-k}\!\left(\frac{N\beta^2 W_k^2}{\rho}\right)$.  Proceeding as we did in the case $k=1$, one could now use equation~\eqref{E: expectation of log non-central chi-squared} to obtain an analogue of equation~\eqref{E: max exponent as series}, featuring the moments of $W_k^2$ in the coefficients of the infinite series. 
Note that  $W_1 \equiv 1$, so this agrees with the previous expression~\eqref{eq:lambda1} for $\lambda_k$ when $k=1$.    

\subsubsection{Singular value decompositions}
For any $k\times k$ real matrix $P$, there exists a \emph{singular value decomposition} $P = ADB$, where $A$ and $B$ are orthogonal $k\times k$ matrices and $D$ is a diagonal matrix  $D = \mathrm{Diag}(d_1, \dots, d_k)$ with $d_1 \ge d_2 \ge \dots \ge d_k \ge 0$. The diagonal factor $D$ is uniquely determined subject to this condition, and the non-negative numbers $d_i$ are called the \emph{singular values} of $P$. (Since $P^TP = B^T D A^T A D B = B^{-1}D^2B$, the $d_i$ are also the positive square roots of the eigenvalues of $P^T P$.)   However, the pair $(A, B)$ is not uniquely determined by $P$ because there are non-trivial orthogonal matrices which commute with $D$. In the case where the singular values of $P$ are distinct, the ambiguity in $B$ consists of a choice of sign for each row of $B$.

We will consider a sequence of singular value decompositions $P(t) = A(t)D(t)B(t)$, where $P(t) = M(t) \dots M(1)$ as above. Recall that we are interested in the asymptotic behaviour of $\Cov^X(t)$. In terms of the singular value decomposition, we have
\begin{equation} \label{E: Cov in terms of SVD}
\begin{split}
\Cov^X(t) &= \frac{1}{N} \bar X(t)^T \bar X(t) = \frac{1}{N} \bar X(0)^T V^T P(t)^T P(t) V \bar X(0)  \\ &= \frac{1}{N} \bar X(0)^T V^T B(t)^T\,D(t)^2B(t) V \bar X(0) \,.
\end{split}
\end{equation}
Note that the final expression does not involve $A(t)$. 

\begin{lemma}\label{L: GuivarchRaugi}
The singular value decompositions $P(t) = A(t) D(t) B(t)$ for $t = 1,2,3, \dots$ may be chosen so that $B(t)$ converges as $t \to \infty$ to a random orthogonal matrix $B \in O(N-1)$. The limit matrix $B$ depends on the choice only up to multiplication of each row separately by $\pm 1$. Up to this row-by-row sign ambiguity, $B$ is distributed according to the Haar measure on $O(N-1)$, and in particular, the first row $B_1$ is distributed (up to sign) according to the uniform measure on the sphere $\mathbb{S}^{N-2}$. Regardless of the choice, $D(t)^{1/t}$ converges to a deterministic diagonal matrix $D$ with distinct positive diagonal entries satisfying $D_{1,1} > \dots > D_{N-1,N-1}$.  Moreover, for $i=1, \dots, N-1$ we have $D_{i,i} = \exp(\lambda_i)$, where $\lambda_1 > \lambda_2 > \dots > \lambda_{N-1}$ are the Lyapunov exponents of the sequence $(P(t))_{t \ge 0}$. 
\end{lemma}
\begin{proof}
We will use the results of Guivarc'h and Raugi \cite{GR} which are also stated in~\cite{BougerolLacroix}, Chapter~4, Theorem~1.2 and Corollary~1.3. The theorem concerns a sequence of random matrix products $S_n = Y_n \dots Y_1$ where $Y_1, Y_2, \dots $ are i.i.d.~real $k$ by $k$ matrices with common distribution $\mu$, and $Y_1$ is almost surely invertible. Denote by $T_\mu$ the smallest closed sub-semigroup of $\mathrm{GL}_{k}(\mathbb{R})$ that contains the support of the distribution of $Y_1$. For any $p \in \{1, \dots, k-1\}$, $T_\mu$ is defined to be \emph{$p$-strongly irreducible} if there does not exist a finite union~$L$ of proper linear subspaces of the exterior power $\bigwedge^p \mathbb{R}^k$ such that, for all $M \in T_\mu$, $\left(\bigwedge^p M \right)(L)$ is contained in $L$. Here $\bigwedge^p M$ denotes the unique linear endomorphism of $\bigwedge^p \mathbb{R}^k$ which maps any decomposable vector of the form $v_1 \wedge \dots \wedge v_p$ to $(Mv_1) \wedge \dots \wedge (M v_p)$. Also for any $p \in \{1, \dots, k-1\}$, $T_\mu$ is defined to be \emph{$p$-contracting} if there exists a sequence $\{M_n: n \ge 1\}$ in $T_\mu$ such that $ \| \bigwedge^p M_n\|^{-1} \bigwedge^p M_n$ converges to a rank one matrix. Guivarc'h and Raugi proved that if $T_\mu$ is $p$-strongly irreducible and $p$-contracting for every $p \in \{1, \dots, k-1\}$ and in addition $\mathbb{E}(\log^{+} \| Y_1\|) < \infty$, (where $\|\cdot\|$ denotes the operator norm with respect to the Euclidean norm), then almost surely the following hold:
\begin{itemize}
\item There exist polar decompositions $S_n = K_n D_n U_n$, with $K_n, U_n \in O(k)$ and $D_n$ diagonal with non-negative entries, such that $K_n \to K$ as $n \to \infty$, where the random limit matrix $K$ is also orthogonal.
\item Taking the positive definite $(2n)^{th}$ root, $\left(S_n^T S_n\right)^{1/2n} \to K^T D K$ where $D$ is a deterministic diagonal matrix such that for each $i \in \{1, \dots, k\}$, $D_{i,i} \ge 0$ and, for $i < d$, if $D_{i,i} > 0$ then $D_{i,i} > D_{i+1,i+1}$. 
We also have $\left(\left(D_n\right)_{i,i}\right)^{1/n} \to D_{i,i}$. 
\end{itemize}
The numbers $\gamma_1, \dots, \gamma_k \in \mathbb{R} \cup \{-\infty\}$ defined by $\gamma_i = \log(D_{i,i})$ are called the \emph{Lyapunov exponents} of the random sequence $\left(S_n\right)$. They satisfy the relation
$$ 
\sum_{i=1}^k \gamma_i = \mathbb{E} (\log | \det Y(1) |)\,,
$$ 
so all $\gamma_i$ are finite (equivalently, all $D_{i,i}$ are positive) if and only if $\mathbb{E}(|\log(|\det Y(1)|)|) < \infty$.

We apply the above result to our random matrix products $P(t) = M(t) \dots M(1)$. We let $P(t) = A(t)D(t)B(t)$ be a singular value decomposition of $P(t)$, chosen so that in each row of $B(t)$ the leftmost nonzero entry is positive. We must check that the conditions of the theorem are satisfied in this case. We have $k=N-1$ and {(thinking of $M(1)$ as a $\mathrm{GL}_k(\mathbb{R})$-valued random variable)} the support of the distribution of $M(1)$ is the whole group $\mathrm{GL}_{k}(\mathbb{R})$, so $T_\mu = \mathrm{GL}_{k}(\mathbb{R})$. 
It is straightforward to check that $T_\mu$ is $p$-strongly irreducible and $p$-contracting for each $p \in \{1, \dots, k-1\}$. The details are spelled out in~\cite{BougerolLacroix}, Chapter~4, Proposition~2.3, {p.~83,} which asserts that if $T_\mu$ contains an open subset of $\mathrm{SL}_k(\mathbb{R})$ then $T_\mu$ is $p$-strongly irreducible and $p$-contracting for each $p \in \{1, \dots, k-1\}$.  It remains to check that $\mathbb{E}(\log^+ \| M(1) \|) < \infty$, which is also straightforward; indeed $\mathbb{E}(\|M(1)\|^2) < \mathbb{E}( \sum_{i,j} M(1)_{i,j}^2) < \infty$.   To see that the Lyapunov exponents are strictly positive, and hence distinct, we must check that $\mathbb{E}(| \log( |\det M(1)| )|) < \infty$. The quantity $\det M(1)$ is a non-constant polynomial function of the $(N-1)^2$ independent non-degenerate Gaussian matrix entries of~$M(1)$, and therefore $\det M(1)$ has a distribution with a continuous density on the real line. (In fact, from~\cite[Section~ 2]{KV} one may read off an explicit recursive description of the distribution of~$\det M(1)$ as a mixture of almost surely non-degenerate Gaussians, with the recursion indexed by the parameter $N$.)  In particular the density of~$\det M(1)$ is bounded on a neighbourhood of $0$, which implies that $\mathbb{E}(|\log^{-}(|\det M(1)|)|) < \infty$, and we also have $\mathbb{E}(\log^{+}(|\det M(1)|)) \le \mathbb{E}((N-1)\log^{+}(\|M(1)\|)) < \infty$ because~$\|M(1)\|$ is the largest singular value of~$M(1)$, while~$|\det(M(1))|$ is the product of the singular values.  

The entire sequence $\left(P(t)\right)_{t \ge 1}$ is invariant in law under conjugation by any member of the orthogonal group~$O(N-1)$. This implies that the limit matrix $B$ must be distributed (up to its row-by-row sign ambiguity) according to Haar measure on~$O(N-1)$, since this is the unique right-invariant Borel probability measure on~$O(N-1)$. Consequently, up to sign, $B_1$ is distributed according to the uniform measure on the unit sphere $\mathbb{S}^{N-2} \subset \mathbb{R}^{N-1}$. Finally, the statements about the dependence of $B$ and $D$ upon the choices of polar decompositions $P(t) = A(t) D(t) B(t)$ follow from the fact that almost surely the singular values of $P(t)$ are distinct for each $t \ge 1$.
\end{proof}

We will denote the Lyapunov exponents of the sequence $\left(P(t)\right)_{t \ge 1}$ by $\lambda_1 > \dots > \lambda_{N-1}$.  That is, letting $D  = \lim_{t \to \infty} D(t)^{1/t}$ be the limit given by Lemma~\ref{L: GuivarchRaugi}, we have $D_{i,i} = \exp(\lambda_i)$, for $i = 1, \dots, N-1$.  

\subsection{Proof of Theorem~\ref{t1}}
\begin{proof}[Proof of Theorem~\ref{t1}]
We already computed the explicit value of the maximal Lyapunov exponent $\lambda_1$ in Section~\ref{SS: max exponent}. By Lemma~\ref{lem: condition for mean opinion to converge}, to prove part (i) of Theorem~\ref{t1} it suffices to show that $\lambda_1 < 0$ implies that $\sum_{t=0}^\infty \Cov^X(t)$ converges almost surely. This follows from Lemma~\ref{L: GuivarchRaugi} and equation~\eqref{E: Cov in terms of SVD}. Indeed, if $\lambda_1 < 0$ then $D_{1,1} = \exp(\lambda_1) < 1$ so $D(t)$ converges exponentially fast to $0$. Since $B(t)$ also converges, \eqref{E: Cov in terms of SVD} implies that $\Cov^X(t)$ almost surely converges exponentially fast to the zero matrix, which implies the convergence of the series $\sum_{t=0}^\infty \Cov^X(t)$.

Part (ii) concerns the critical case where  $\lambda_1 = 0$. We saw in Lemma~\ref{L: empirical variances of single coordinates are random walks} that $\Cov^X_{1,1}(t)$ is a random walk with i.i.d.~increments distributed as the log of a non-central $\chi^2$ random variable, which has positive finite variance and mean $0$ in the critical case. This random walk is almost surely unbounded above. This follows, for example, from Theorem 5.1.7 in~\cite{DL}. In particular $\sum_{t=0}^\infty \Cov^X(t)$ almost surely diverges. It then follows from Lemma~\ref{lem: condition for mean opinion to converge} that $\mu(t)$ almost surely diverges as $t \to \infty$. 

Suppose $\lambda_1 > 0$. Let $\mathcal{G}$ be the $\sigma$-algebra generated by $(\overline{X}(t))_{t \ge 0}$. 
Let $B = \lim_{t \to \infty} B(t)$ be the $\mathcal{G}$-measurable orthogonal matrix described in Lemma~\ref{L: GuivarchRaugi}, where (to make it well-defined) the choice of signs is made so that the first non-zero entry in each row of $B$ is positive. We will study the sequence of dot products $u \cdot X_i(t)$.
We will use the conditional L\'evy-Borel-Cantelli lemma to show that almost surely, conditional on~$\mathcal{G}$, we have $|u \cdot X_i(t)| \to \infty$ almost surely for each $i=1, \dots, N$. Conditional on $(\mathcal{G}, X(t-1))$, the distribution of $\mu(t)$ is Gaussian with mean $\mu(t-1)$ and covariance $\frac{\rho}{N} \Cov^X(t-1)$, and the distribution of each $X_i(t)$ is Gaussian with mean $\overline{X}_i(t)$ and covariance $\frac{\rho}{N} \Cov^X(t-1)$. (Indeed, conditional on $\mathcal{G}$ each of $X_1(t), \dots, X_N(t)$ and $\mu(t)$ determines the others.) It follows that conditional on $(\mathcal{G}, X(t-1))$, the distribution of $u \cdot X_i(t)$ is Gaussian. We now compute its variance using equation~\eqref{E: Cov in terms of SVD}:
\begin{eqnarray*} \mathrm{Var}(u \cdot X_i(t) \,|\, \mathcal{G}, X(t-1)) &=& \rho\, u\, \Cov^X\!(t-1)\, u^T \\
 &  = & \frac{\rho}{N} u\, \bar X(0)^T V^T B(t-1)^T\,D(t-1)^2B(t-1) V \bar X(0)\,  u^T\,.
\end{eqnarray*}
By Lemma~\ref{L: GuivarchRaugi} {we have $D_{i,i}(t)^{1/t} \to \exp(\lambda_i)$ as $t \to \infty$, so} the matrix $D(t)/D_{1,1}(t)$ converges almost surely to the {$N-1$ by $N-1$} matrix whose entries are all $0$ except for a $1$ in the top-left corner. Also $B(t)$ converges a.s.~to the $\mathcal{G}$-measurable random limit $B$ in $O(N)$; {recall that $B_1$ is the first row of $B$}. Hence a.s. we have 
$$\frac{1}{D_{1,1}(t-1)^2} \mathrm{Var}(u \cdot X_i(t) \,|\, \mathcal{G}, X(t-1)) \; \to \; \frac{\rho}{N } \left(u \cdot (B_1 V \bar X(0)) \right)^2\,.$$  Let $W$ denote the $\mathcal{G}$-measurable random variable $|u \cdot (B_1 V \bar X(0))|$, which is the absolute value of the sole entry of the $1$ by $1$ matrix $u \bar X(0)^T V^T B_1^T$. The assumption that $\mathrm{Cov}^X(0)$ is a.s.~positive definite implies that the column vector $u \bar X(0)^T V^T$ is a.s.~nonzero.  Recall that the law of $B_1$ is absolutely continuous with respect to the uniform measure on the sphere $\mathbb{S}^{N-2}$ (up to its sign ambiguity). Since $B$ is independent of $X(0)$, we deduce that $W \neq 0$ almost surely. 
For any $R> 0$ we may bound the Gaussian density by its maximum and integrate over $(-R,R)$ to obtain
$$\mathbb{P}(|u \cdot X_i(t)| < R \,|\, \mathcal{G}, X(t-1)) \le \frac{ 2 R}{\left(2\pi \mathrm{Var}(u \cdot X_i(t) \,|\, \mathcal{G}, X(t-1))\right)^{1/2}}\, \sim\, \frac{2 \sqrt{N}\, R }{\sqrt{2\pi\rho} \,W D_{1,1}(t-1)}$$ For any $\epsilon > 0$ we have
$$\frac{2\sqrt{N} R }{\sqrt{2\pi\rho} \,W D_{1,1}(t-1)}\,=\, O\left(\frac{2\sqrt{N}\, R\, e^{-(\lambda_1-\epsilon) (t-1)}}{\sqrt{2\pi\rho}\,W }\right)\,, $$
almost surely.  The RHS above is summable in $t$ when $\epsilon < \lambda_1$, hence 
$$\mathbb{P}( | u \cdot X_i(t) | < R\; \text{ for infinitely many $t$ }) = 0\,. $$
Since this holds for each $R>0$, we have $|u \cdot X_i(t)| \to \infty$ almost surely.

Essentially the same argument applies to $u\cdot \mu^X(t)$, using 
$$\mathrm{Var}(u \cdot \mu^X(t) \,|\, \mathcal{G}, X(t-1)) = \frac{\rho}{N} u\, \mathrm{Cov}^X\!(t-1)\, u^T\,. $$
\end{proof}

\subsection{Proof of Theorem~\ref{t2}}

\begin{proof}[Proof of part (a) of Theorem~\ref{t2}]
We will apply a theorem of Guivarc'h and Raugi \cite{GR}, which is given in~\cite[Ch.~4, Thm.~1.2]{BougerolLacroix}. Let $k \in \{1, \dots, N-1\}$. Recall that in the proof of Lemma~\ref{L: GuivarchRaugi} we checked that the random matrix product sequence $\left(P(t)\right)_{t \ge 0}$ is $k$-strongly irreducible and $k$-contracting.  Guivarc'h and Raugi showed that these conditions imply that there is a unique invariant measure for the action of the random matrices $M(1), M(2), \dots$ on the projective space {$P\left(\bigwedge^k \mathbb{R}^{N-1}\right)$\,}. We will explain in detail what this means\footnote{Although we are simply extending the computation of Newman \cite{Newman}, the details of this argument were not included in Newman's paper, and we are aware of some papers on Lyapunov exponents in the mathematical physics literature in which similar calculations are performed with incorrect justification.}.

The sum of the first $k$ Lyapunov exponents is the maximal Lyapunov exponent for the induced action of the matrix product sequence $\left(P(t)\right)_{t \ge 0}$ on the $k^{\textup{th}}$ exterior power~$\bigwedge^k \mathbb{R}^{N-1}$. This exterior power is the~$\binom{N-1}{k}$-dimensional Hilbert space with orthonormal basis \begin{equation}\label{orthonormal basis} (e_{i_1} \wedge \dots \wedge e_{i_k}\,:\, 1 \le i_1 < \dots < i_k \le N-1)\,.\end{equation}  Any linear map $L: \mathbb{R}^{N-1} \to \mathbb{R}^{N-1}$ induces a linear map $\hat{L}:  \bigwedge^k \mathbb{R}^{N-1} \to \bigwedge^k \mathbb{R}^{N-1}$ which acts on basis elements by $$\hat{L}: e_{i_1} \wedge \dots \wedge e_{i_k} \mapsto \left(Le_{i_1}\right) \wedge \dots \wedge \left(Le_{i_k}\right)$$ for each tuple $(i_1, \dots, i_k)$ such that $1 \le i_1 < \dots < i_k \le N-1$. 

To express $\left(Le_{i_1}\right) \wedge \dots \wedge \left(Le_{i_k}\right)$ in terms of the orthonormal basis~\eqref{orthonormal basis}, one expands the $k$-fold wedge using multilinearity together with the antisymmetry relation, which says that for any permutation $\sigma$ in the symmetric group $S_k$, and any $1 \le i_1 \le \dots\le i_k \le N-1$ we have {
 $$
 e_{i_{\sigma(1)}} \wedge \dots \wedge e_{i_{\sigma(k)}} = \mathrm{sgn}(\sigma) \,e_{i_1} \wedge \dots \wedge e_{i_k}\,.
 $$ }
(In particular, any $k$-fold wedge with repetition is $0$, which is why such wedges are not included in the basis.)  One can check that if $L_1$ and $L_2$ are two linear maps from $\mathbb{R}^{N-1}$ to $\mathbb{R}^{N-1}$ then $\widehat{(L_1 L_2)} = \hat{L}_1 \hat{L}_2$. The induced action of the identity map is the identity, so $\widehat{(L^{-1})} = (\hat{L})^{-1}$. In particular, 
 $$ \widehat{P(t)} = \widehat{M(t)}\widehat{M(t-1)} \dots \widehat{M(1)}\,,$$ and $\widehat{P(t)}$ is almost surely invertible.
 
Any invertible linear map on $\bigwedge^k \mathbb{R}^{N-1}$ descends to a self-map of the projective space $P\left(\bigwedge^k \mathbb{R}^{N-1}\right)$. The theorem of Guivarc'h and Raugi mentioned above states that there is a unique invariant probability measure~$\pi_k$ for this projective action. Let $\widehat{O(N-1)}$ denote the group of transformations of $P\left(\bigwedge^k \mathbb{R}^{N-1}\right)$ induced by the action of $O(N-1)$ on $\mathbb{R}^{N-1}$. Since $\pi_k$ is unique, it must be invariant with respect to $\widehat{O(N-1)}$, because the matrices $M(t)$ are invariant in law under conjugation by $O(N-1)$ and hence the matrices $\widehat{M(t)}$ are invariant in law under conjugation by the induced orthogonal transformations of $\bigwedge^p \mathbb{R}^{N-1}$. Furthermore, the measure~$\pi_k$ is supported on the closed subvariety $\overline{D} \subseteq P\left(\bigwedge^k \mathbb{R}^{N-1}\right)$ which is the image under projectivization of the set~$D$ of  \emph{decomposable} vectors in $\bigwedge^k \mathbb{R}^{N-1}$. A non-zero vector in $\bigwedge^k \mathbb{R}^{N-1}$ is called decomposable if it is of the form $v_1 \wedge \dots \wedge v_k$, where $v_1, \dots, v_k$ are any $k$ linearly independent vectors in $\mathbb{R}^{N-1}$. 

Let $M$ be a random matrix distributed like $M(1)$, and let $\hat{M}$ be the map that $M$ induces on $\bigwedge^k\mathbb{R}^{N-1}$.
The action of $\widehat{O(N-1)}$ on $\overline{D}$ is transitive, and this implies that for any two points $x, y \in D$, the distributions of $\frac{\|\hat{M}x\|}{\|x\|}$ and $\frac{\|\hat{M}y\|}{\|y\|}$ coincide. Indeed, there exists some $U \in O(N-1)$ such that $\hat{U}(x) \propto y$, and the law of $M$ is invariant under conjugation by $U$, so the law of $\hat{M}$ is invariant under conjugation by $\hat{U}$.

 Let $M_1^k$ denote the $N-1$ by $k$ matrix whose columns are the first $k$ columns of $M$. Applying Furstenberg's integral formula to compute the maximal Lyapunov exponent of the sequence $\left(\widehat{P(t)}\right)_{t \ge 1}$, we find that 
\begin{eqnarray*} 
\lambda_1(\alpha,\beta,N) + \dots + \lambda_k(\alpha,\beta,N) & = & \mathbb{E}\left(\log \frac{\|\hat{M}(e_1 \wedge \dots \wedge e_k)\|}{\|e_1 \wedge \dots \wedge e_k\|}\right)\\ & = &  \mathbb{E}\left(\log \|M(e_1 )\wedge \dots \wedge M(e_k)\|\right)\\ 
 & = & \frac{1}{2} \mathbb{E}\left(\log \det \left(\left(M_1^k\right)^T M_1^k \right)\right)\,. 
\end{eqnarray*}
  (Note that in the top-dimensional case $k=N-1$ this reduces to the usual determinant formula $\mathbb{E} \log (|\det M|)$ for the sum of all the Lyapunov exponents.) Comparing the above formulae for $k$ and $k-1$ (for any $2 \le k \le N-1$), we obtain
\begin{equation}\label{E: matrix expression for lambda k}
 \lambda_k(\alpha,\beta,N) = \frac{1}{2} \mathbb{E}\left(\log \frac{\det \left(\left(M_1^k\right)^T M_1^k\right)}{\det \left(\left(M_1^{k-1}\right)^T M_1^{k-1} \right) }\right)\,. 
\end{equation}
(Note that the determinants above are related to the volume of the relevant parallelepipeds.)

The ratio whose logarithm appears in the above formula is simply the squared length of the projection of row $k$ of~$M$ onto the orthogonal complement of the span of the first $k-1$ rows of~$M$. Row $k$ is the
sum of two independent vectors, namely $\beta e_k$ and $\frac{\sqrt{\alpha}(1-\beta)}{\sqrt{N}}G_k$, where $G_k$ is a mean zero Gaussian vector whose covariance matrix is~$I_{N-1}$. 
Let $S_{k-1}$ be the span of the first $k-1$ rows of $M$, and let $\pi_{k-1}$ denote the orthogonal projection onto the orthogonal complement of $S_{k-1}$. {Consider the three mutually orthogonal subspaces $S_{k-1}$, $\langle \pi_{k-1}(e_k) \rangle$, and $(S_{k-1} + \langle e_k \rangle )^\perp$, of dimensions $k-1$, $1$ and $N-k-1$. By projecting $G_k$ onto these three subspaces, we may express $G_k$ as the sum of three mutually orthogonal vectors $G_k = v_1 + v_2 + v_3$, where $v_1 \in S_{k-1}$ and $v_2$ is a scalar multiple of $\pi_{k-1}(e_k)$. Row $k$ of $M$ is a Gaussian vector independent of~$S_{k-1}$, with covariance a multiple of the identity. So conditional on $S_{k-1}$, the orthogonal summands $v_1$, $v_2$ and $v_3$ are conditionally independent Gaussian vectors. Moreover, conditional on $S_{k-1}$, $\|v_2\|^2$ has a $\chi^2(1)$ distribution, and $\|v_3\|^2$ has a $\chi^2(N-k-1)$ distribution, and these squared lengths are conditionally independent. Since these distributions do not depend on $S_{k-1}$, it  also holds unconditionally that $\|v_2\|^2$ and $\|v_3\|^2$ are independent with these two $\chi^2$ distributions (even though $v_2$ and $v_3$ are dependent).} Finally, $v_2$ is parallel or antiparallel to $\pi_{k-1}(e_k)$ with equal probability conditional on $\|v_2\|$. Recall that $\rho = \alpha(1-\beta)^2$. 
We obtain
\begin{eqnarray}\label{E: geometry}\frac{\det \left(\left(M_1^k\right)^T M_1^k\right)}{\det \left(\left(M_1^{k-1}\right)^T M_1^{k-1}\right) } & = & \left\| \left(\frac{\sqrt{\alpha}(1-\beta)}{\sqrt{N}}v_2 + P_{k-1}(\beta e_k)\right) + \frac{\sqrt{\alpha}(1-\beta)}{\sqrt{N}}v_3\right\|^2 \notag \\ & \stackrel{d}{=} & \left( \frac{\sqrt{\rho}}{\sqrt{N}} X_1 + \beta W_k\right)^2 + {\frac{\rho}{N}}\sum_{i=1}^{N-k} X_i^2\,,\end{eqnarray}
where $X_1, \dots, X_{N-k}$ are independent standard Gaussian random variables, independent of $W_k$,  and $W_k$ is the random variable distributed as $\|P_{k-1}(e_k)\|$. Combining equations~\eqref{E: matrix expression for lambda k} and~\eqref{E: geometry} yields the desired formula for~$\lambda_k(\alpha,\beta,N)$.
\end{proof}

\vspace{1cm}
\begin{proof}[Proof of part (b) of Theorem~\ref{t2}]
For $1 \le k \le N-2$, note that $W_k$ strictly stochastically dominates $W_{k+1}$. To see this, observe that $W_k$ is also the distribution of $P_{k-1}(e_{N-1})$, since the law of $S_{k-1}$ is invariant under the orthogonal transformation which exchanges coordinates $k$ and $N-1$. Note that $P_k = P_k P_{k-1}$ so $P_{k}(e_{N-1}) = P_{k}(P_{k-1}(e_{N-1})) $.
Define
$$
V_k = \left(\beta W_k + \frac{\sqrt{\alpha}(1-\beta)}{\sqrt{N}}X_1\right)^2 + \sum_{i=2}^{N-k}\frac{\alpha(1-\beta)^2}{N} X_i^2 \,.
$$
It follows that $\lambda_{k+1} < \lambda_k$, because $\lambda_k = \mathbb{E}(\log V_k)$ and $\lambda_{k+1} = \mathbb{E}\left(\log V_{k+1}\right)$, for coupled random variables such that $V_k > V_{k+1}$ almost surely.  
Note also that $0 \le W_k \le 1$ almost surely. Let 
$$
E_k=\{X_1 < 0
\text{ and }|X_{N-k}|=\max(|X_1|, \dots, |X_{N-k}|)\ge 1\}.
$$
Note that
\begin{itemize}
\item $\mathbb{P}(X_1<0)=1/2$ and the sign of $X_1$ is independent of the absolute values $|X_1|, \dots, |X_N|$,
\item 
$\mathbb{P}(\max(|X_1|,\dots,|X_{N-k}|)\ge 1)\ge \mathbb{P}(|X_1| \ge 1)= 2\,\mathbb{P}(X_1 \ge 1) > 0.3$, 
\item the events $|X_{N-k}|=\max(|X_1|,\dots,|X_{N-k}|) $ and $\max(|X_1|,\dots,|X_{N-k}|)  \ge 1 $ are independent,
\item 
by symmetry, $\mathbb{P}(|X_{N-k}| = \max(|X_1|, \dots, |X_{N-k}|)) = 1/(N-k)$.
\end{itemize} 
Therefore
\begin{align*} 
\mathbb{P}\left(E_k\right)  & =  \mathbb{P}(X_1 < 0,\ |X_{N-k}| \ge 1 \;\text{ and }\;|X_{N-k}| = \max(|X_1|, \dots, |X_{N-k}|))  
\\ & >  \frac12 \times 0.3 \times \frac1{N-k}=\frac{0.15}{N-k}\,.
\end{align*} 
Note that $0 \le W_k \le 1$ a.s. and on the event $X_1 < 0$ we have
$$
\left(\beta W_k + \frac{\sqrt{\alpha}(1-\beta)}{\sqrt{N}} X_1\right)^2 \le \beta^2 W_k^2 + \frac{\alpha(1-\beta)^2}{N} X_1^2 \le \beta^2 + \frac{\alpha(1-\beta)^2}{N} X_1^2\,.
$$ 
Therefore on the event $E_k$ we have
\begin{align*}
\frac{V_{k+1}}{V_k} &= 1 - \frac{\frac{\alpha(1-\beta)^2}{N}X_{N-k}^2}{\left(\beta W_k + \frac{\sqrt{\alpha}(1-\beta)}{\sqrt{N}}X_1\right)^2 + \sum_{i=2}^{N-k}\frac{\alpha(1-\beta)^2}{N} X_i^2} \le 1 - \frac{\frac{\alpha(1-\beta)^2}{N}X_{N-k}^2}{\beta^2  + \sum_{i=1}^{N-k}\frac{\alpha(1-\beta)^2}{N} X_{N-k}^2} \\ & \le 1 - \frac{1}{\frac{N\beta^2}{\alpha(1-\beta)^2} + N-k}.
\end{align*}
Since $V_{k+1} \le V_k$ surely, and for $0 < x \le 1$ we have $\log(1-x) \le -x$, we find
$$\mathbb{E}(\log(V_{k+1}/V_k)) \le \frac{-0.15}{(N-k)\left(\frac{N\beta^2}{\alpha(1-\beta)^2} + N-k\right)}\,, $$
which is to say that
$$\lambda_{k}- \lambda_{k+1} \ge \frac{0.15}{(N-k)\left(\frac{N\beta^2}{\alpha(1-\beta)^2} + N-k\right)}\,,$$ for every $k$ in the range $1 \le k \le N-2$. 
Notice that this lower bound on the gap between successive Lyapunov exponents is asymptotic to  $0.15\alpha(1-\beta)^2/(N(N-k))$ as $\beta$ approaches $1$ from below. 
\end{proof}

\begin{proof}[Proof of Theorem~\ref{t3}]
Since $\Cov^X(t)$ is a non-negative definite symmetric matrix, its operator norm is its largest eigenvalue.  Then the fact that $D_{1,1} > D_{2,2}$ implies that the right-hand side of equation~\eqref{E: Cov in terms of SVD} is dominated asymptotically by the contribution of the top-left entry of $D(t)$, and hence $$\frac{ \Cov^X(t)}{\|\Cov^X(t)\|} \; \to \; \frac{C}{\|C\|}, \;\text{ where }\;\; C = \bar{X}(0)^T\,V^TB_1^T\,B_1V\bar{X}(0)\,.$$ The limit above is a random $d$ by $d$ matrix of rank one (since $B_1$ has rank one), as long as $C \neq 0$. It remains to check that $C \neq 0$ almost surely. Indeed, $C$ could only be zero if $B_1 V \overline{X}(0) = 0$, but we assumed that the initial opinions are almost surely not all equal, which implies that $V\overline{X}(0) \neq 0$; then since $B_1$ is uniformly distributed on the unit sphere $\mathbb{S}^{N-2}$, up to its sign ambiguity, the probability that $B_1$ lies in the left kernel of~$V\overline{X}(0)$ is zero. 

We defer the proof of equation~\eqref{eq:t3} to the next section; see Lemma~\ref{lem:lam1lam2} below.
\end{proof}

\subsection{Some properties of the Lyapunov exponents \m{\lambda_i(\alpha,\beta,N)}}\label{S: some properties}

In the next statement, we use the exponential integral defined for $x>0$ as
$$
\mathrm{Ei}_1(x)=\int_1^\infty \frac{e^{-xu}}{u}du=\int_0^x \frac{1-e^{-u}}{u}du-\gamma- \log(x).
$$
\begin{proposition}\label{prop:Nodd}
Suppose that $N$ is odd and $\beta>0$. Let $m=(N-1)/2$. Then
 \begin{align}\label{E: analytic expression for lambda1}
\lambda_1(\alpha,\beta,N)&=\log\beta
+\frac12\left( \mathrm{Ei}_1 \left(z\right)
+ \sum_{i=1}^{m-1} \frac{(-1)^i\, (i-1)!}{z^i}
\left[e^{-z}-\binom{m-1}{i}\right]\right)
\end{align}
where $z=\frac{N\beta^2}{2\rho}=\frac{N\beta^2}{2\alpha (1-\beta)^2}$.
\end{proposition}
\begin{lemma}\label{lem:lambda1asympt}
As $N\to\infty$ with $\alpha$, $\beta$ and $\rho  = \alpha(1-\beta)^2$ all fixed,
\begin{align}\label{lam1asy}    
\lambda_1(\alpha,\beta,N)
&=\frac12%\left[
\log \left(\rho+\beta^2\right) 
+o(1).
%-\frac{\rho}N\cdot \frac{2\rho+3\beta^2}{(\rho+\beta^2)^2}+O\left(N^{-2}\right)\right].
\end{align}
\end{lemma}
\begin{proof}
Set $z=Ny/2$, where $y=\frac{\beta^2}{\rho}=\frac{\beta^2}{\alpha(1-\beta)^2}=O(1)$ and $y > 0$. Recalling that $m=(N-1)/2$,  from Theorem~\ref{t1} we have
\begin{align*}
2\lambda_1&=\phi_{\frac{N-1}2}\left(\frac{N\beta^2}{2\rho}\right)+\log \frac{2\rho}N\\
&=\psi\left(\frac{N-1}2\right)
+\int_0^1 \frac{\left(1-e^{-\frac{Nys}2}\right)(1-s)^{\frac{N-3}2}\mathrm{d}s}s
+\log\frac{2\rho}N
=\log \rho +o(1)
+I
\end{align*}
where
$$
I=\int_0^1 \frac{\left(1-e^{-\frac{Nys}2}\right)(1-s)^{\frac{N-3}2}\mathrm{d}s}s=
  \int_0^N \frac{(1-e^{-\frac{yt}2})\left(1-\frac{t}N\right)^{\frac{N-3}2}\mathrm{d}t}t 
$$
since $\psi(x)=\log x+ O(1/x)$ as $x \to \infty$, see \cite[6.3.21]{AS}.
%Now,
%\begin{align*}
% I:=\int_0^1 \frac{(1-e^{-\frac{Nys}2})(1-s)^{\frac{N-3}2}\mathrm{d}s}s   &=
%  \int_0^N \frac{(1-e^{-\frac{yt}2})\left(1-%\frac{t}N\right)^{\frac{N-3}2}\mathrm{d}t}t .
%\end{align*}
As $\ln(1-x)\le -x$ for all $\in[0,1)$, 
$$
 \frac{N-3}2\log\left(1-\frac tN\right)
\le
-\frac{t}2\left(1-\frac3N\right)\,.
$$
Since $\ln(1-x)\ge -x-x^2$ for $0 \le x \le \frac{1}{2}$, when $t \le N/2$ we also have
$$
\frac{N-3}2\log\left(1-\frac tN\right)\ge
\frac{N-3}2\left(-\frac tN-\frac {t^2}{N^2}\right)
\ge -\frac t2\,.
$$
Let 
$$
I_*(k,M)=\int_0^M \left(1-e^{-\frac{yt}2}\right)e^{-\frac{kt}2}\frac{\mathrm{d}t}t .
$$
Then
\begin{align*}
 I_*(1-3/N,N)\ge 
I& \ge \int_0^{N/2}\left(1-e^{-\frac{yt}2}\right)\left(1-\frac tN\right)^{\frac{N-3}{2}}\frac{\mathrm{d}t}t 
\ge I_*(1,N/2).
\end{align*}
On the other hand, if $I_*(k,\infty)=\lim_{M\to\infty} I_*(k,M)$, then $0\le I_*(k,\infty)-I_*(k,M)\le \frac2k e^{-kM/2}$ and $I_*(k,M)$ is continuous in $k$, therefore
$$
I=\int_0^\infty \left(1-e^{-\frac{yt}2}\right)e^{-\frac{t}2}\frac{\mathrm{d}t}t +o(1) 
$$
as $N\to\infty$. The integral $I$ can be computed by differentiating it with respect to $y$: if 
$$
f(y)=\int_0^\infty \left(1-e^{-\frac{yt}2}\right)e^{-\frac{t}2}\frac{\mathrm{d}t}t
=\int_0^\infty \left(1-e^{-yt}\right)e^{-t}\frac{\mathrm{d}t}t
$$ then
$$
f'(y)=\int_0^\infty
e^{-(y+1)t}\mathrm{d}t=\frac1{y+1}, \quad f(0)=0.
$$
So $I=f(y) + o(1)=\ln(y+1) + o(1)$,
yielding~\eqref{lam1asy}.
\end{proof}
\begin{remark} By differentiating $I_*(k,\infty)$ with respect to $k$, bounding the derivative, and integrating again, one can obtain an error bound of the form $O(y/N) + O(e^{-N/2})$, where the implied constants do not depend on $y$.
\end{remark}
\begin{remark}
For large $N$ and fixed $\beta$, this expression for $\lambda_1$ equals zero when
$$
\alpha=\alpha_\mathrm{cr}(\beta)= \frac{1+\beta}{1-\beta}+o(1).%+O\left(N^{-1}\right),
$$
where $\alpha_\mathrm{cr}=\alpha_\mathrm{cr}(\beta)$ is defined as the value of $\a$ which makes $\lambda_1$ equal to zero. 
\end{remark}
%%%%
%%%%%
%%%%%%

\begin{lemma}\label{lem:lam1lam2}
As $N\to\infty$,
\begin{align}\label{eq:lam1lam2}
\lambda_1-\lambda_2=\frac1{2N}
\left(1-\left[\frac{\beta^2}{\rho+\beta^2}\right]^2\right)+o\left(\frac1N\right) 
\end{align}
(in particular, if $\beta=0$, we get $\lambda_1-\lambda_2=\frac{1+o(1)}{2N}$.)
\end{lemma}
\begin{proof}
From Theorem~\ref{t2} we have
$2\lambda_2=\E\log Q\,,$ 
where
\begin{align}\label{eq:Q}
Q=\left(\beta W_2 + \frac{\sqrt{\rho}}{\sqrt{N}}X_1\right)^2 + \sum_{i=2}^{N-2}\frac{\rho}{N} X_i^2,    
\end{align}
where $X_1, \dots, X_{N-2}$ are i.i.d.~standard normal. It is easy to verify that 
\begin{align}\label{eq:W_2}
W_2^2=1-\frac{\rho}{N}\cdot
\frac{\xi_2^2}{\left(\beta+\sqrt{\frac{\rho}{N}}\xi_1\right)^2
+\frac{\rho}{N}\left(\xi_2^2+\xi_3^2+\dots+\xi_{N-1}^2\right)}
\in [0,1]
\end{align}
where $\xi_1, \dots, \xi_{N-1}$ are i.i.d.\ standard normal, independent of $X_1, \dots, X_{N-2}$.
{Indeed, the first row of $M(1)$ can be written as $\left(\beta + \sqrt{\frac \rho N}\xi_1,\sqrt{\frac \rho N}\xi_2,\dots,\sqrt{\frac \rho N}\xi_{N-1}\right)$, and the projection of $e_2$ on the linear space spanned by this vector can be obtained by minimizing 
$$
\left[k\left(\beta + \sqrt{\frac \rho N}\xi_1 \right)\right]^2
+\left[k\sqrt{\frac \rho N}\xi_2-1\right]^2
+\left[k\sqrt{\frac \rho N}\xi_3\right]^2
+\dots
+\left[k\sqrt{\frac \rho N}\xi_{N-1}\right]^2
$$
with respect to $k\in\R$.
%... giving k = \sqrt{\rho/N}\xi_2/((\beta + \sqrt{\rho/N})^2 + (\rho/N)(\xi_2^2 + \dots + \xi_{N-1}^2))
Then, using the Pythagoras theorem, we obtain~\eqref{eq:W_2}.}
Hence
\begin{align}
\label{eq:lambda2eta}
2\lambda_2-\log(\rho+\beta^2)=\E\log Q-\log(\rho+\beta^2)=\mathbb{E}\log\left(
1+\eta\right)\,,    
\end{align} 
where 
$$
\eta=\frac{\beta^2 (W_2^2-1) + 
2\beta W_2\frac{\sqrt{\rho}}{\sqrt N} X_1
+ \frac{\rho}{N}X_1^2+\frac{\rho}{N}\left[\sum_{i=2}^{N-2} X_i^2 -N\right]}
{\rho+\beta^2}.
$$
For the rest of the calculations, we assume that $N$ is large. Define the events 
\begin{align*}
\begin{array}{clrl}
A_0&=\left\{|\xi_3^2+\dots+\xi_{N-1}^2-N|<N^{5/8}\right\}, &
A_1&=\left\{|\xi_1|<N^{1/8},\ |\xi_2|<N^{1/8}\right\}, 
\\
B_0&=\left\{|X_2^2+\dots+X_{N-2}^2-N| <N^{5/8}\right\}, &B_1&=\left\{|X_1|<N^{1/8}\right\},
\end{array}
\end{align*}
and let $\mathcal{A}=A_0\cap A_1\cap B_0\cap B_1$. 
The  Massart and Laurent bounds (Lemma~1 in~\cite{LM}) for the chi-squared distribution~$\chi$ with~$k$ degrees of freedom state that
\begin{align*}
\mathbb{P}\left(\chi-k>2\sqrt{kx}+2x\right)\le e^{-x};\qquad
\mathbb{P}\left(\chi-k<-2\sqrt{kx}\right)\le  e^{-x}.
\end{align*}
Hence, noting that the events in $A_0$ and $B_0$ deal with chi-squared random variables with $N-3$ degrees of freedom, and using the trivial bounds for standard normal random variables, we conclude (set $x=N^{1/4}$) that
\begin{align}
\label{eq:Acal}
\mathbb{P}(\mathcal{A}^c)\le c_1 \exp\left(-c_2 N^{1/4}\right)   
\end{align}
for some $c_1,c_2>0$.

On $\mathcal{A}$ we have $|W_2^2-1|\le N^{-3/4}(1+o(1))$ and thus $\eta=O(N^{-3/8})$. As a result,
$$
\log(1+\eta)=\eta-\frac{\eta^2}{2}+c_{\eta}\eta^3
$$
where we can assume that $0\le c_\eta\le 1$.
Consequently,
$$
\E\left[\log(1+\eta) \1_{\mathcal{A}}\right]=\E[\eta\1_{\mathcal{A}}]
-\frac{\E[\eta^2\1_{\mathcal{A}}]}{2}+O(N^{-9/8})
$$
We have
\begin{align*}
(\rho+\beta^2)\E[\eta\1_{\mathcal{A}}]
&=
-\frac{\rho\beta^2}{N} 
\E\left[ 
\frac{\xi_2^2 \, \1_{\mathcal{A}}}{(\beta+\sqrt{\rho/N}\xi_1)^2+\rho \xi_2^2/N+\rho(\xi_3^2+\dots+\xi_{N-1}^2)/N  }\right]
\\ & +
\E\left[ 2\beta W_2\frac{\sqrt{\rho}}{\sqrt N} X_1
\1_{\mathcal{A}}\right]
+ \E\left[ \frac{\rho}{N}\left(\sum_{i=1}^{N-2} X_i^2 -N\right)
\1_{\mathcal{A}}\right]
\\ & =-\frac{\rho\beta^2}{N} \cdot \frac{1+o(1)}{\rho+\beta^2}
+0-\frac{2\rho+o(1)}{N}=-\frac{\rho}{N}\cdot\frac{2\rho+3\beta^2}{\rho+\beta^2}+o\left(\frac1N\right)
\end{align*}
since $X_1$ is symmetric even on $\mathcal{A}$, and 
$\E\left[ \left(\sum_{i=1}^{N-2} X_i^2 -N\right) \1_{\mathcal{A}^c}\right]$ and $ \E\left[ \xi_2^2\1_{\mathcal{A}^c}\right]$ are exponentially small.
Further,
\begin{align*}
(\rho+\beta^2)^2\E[\eta^2\1_{\mathcal{A}}]
&=
\E\left[
\left(
\beta^2(W_2^2-1) + 
2\beta W_2\frac{\sqrt{\rho}}{\sqrt N} X_1
+ \frac{\rho}{N}X_1^2+\frac{\rho}{N}(Y-3)
\right)^2
\1_{\mathcal{A}}\right]
\\
&=
\frac{4\beta^2\rho+o(1)}{N}
+0+\E\left[
\left(
\beta^2(W_2^2-1) + \frac{\rho}{N}X_1^2+\frac{\rho}{N}(Y-3)\right)^2\1_{\mathcal{A}}\right]
\\ &=
\frac{4\beta^2\rho}{N}
+\frac{\rho^2}{N^2}\E\left[Y^2\1_{\mathcal{A}}\right]
+o\left(\frac1N\right) 
=
\frac{4\beta^2\rho}{N}
+\frac{2\rho^2(N-3)}{N^2}
+o\left(\frac1N\right) 
\\ 
&=
\frac{\rho(2\rho+4\beta^2)}{N}
+o\left(\frac1N\right) 
\end{align*}
where $Y=\sum_{i=2}^{N-2} X_i^2 -(N-3)$ is centred chi-squared random variable with $N-3$ degrees of freedom,  using again the fact that $X_1$ is symmetric, and 
$\E\left[Y^2\1_{\mathcal{A}^c}\right]$  is exponentially small. Summarizing,
$$
\E\left[\log(1+\eta) \1_{\mathcal{A}}\right]=
-\frac{\rho}{N}\left[\frac{2\rho+3\beta^2}{(\rho+\beta^2)^2}+
\frac{\rho+2\beta^2}{(\rho+\beta^2)^2}
\right]+o\left(\frac1N\right) 
=
-\frac{\rho}{N}\frac{3\rho+5\beta^2}{(\rho+\beta^2)^2}+o\left(\frac1N\right) 
$$
On the other hand, by the Cauchy-Schwarz inequality and~\eqref{eq:Acal} we obtain
\begin{align}\label{eq:CS}
 \left|\E \left[\log(Q)\, \1_{\mathcal{A}^c}\right]\right|
 &\le\sqrt{\E\log^2 (Q)}
 \sqrt{\E\left[ \1_{\mathcal{A}^c}\right]}
 =\sqrt{\E\log^2(Q)}\cdot \sqrt{\P(\mathcal{A}^c)}
 \nonumber
\\ &\le \sqrt{\E\log^2(Q)} \cdot c_1^{1/2} \exp\left(-c_2N^{1/4}/2\right)
\end{align}
where $Q$ is defined by~\eqref{eq:Q}. Since
\begin{align*}
\E \left[\log^2(Q)\, \1_{Q\ge 1}\right]&\le
\E \left[\log^2(Q)\, \1_{Q\ge 1}\right] \le \E[ Q^2]=O(1);
\\
\E \left[\log^2(Q)\, \1_{Q< 1}\right]&\le
\E \left[Q^{-1}\, \1_{Q< 1}\right] \le \E [Q^{-1}]
\le \frac{N}{\rho}\E\left[\frac{1}{X_2^2+\dots+X_{N-2}^2}
\right]=\frac{N}{\rho(N-5)}=O(1),
\end{align*}
equations~\eqref{eq:lambda2eta} and~\eqref{eq:CS} imply that
$\E\left[\log(1+\eta) \1_{\mathcal{A}^c}\right]=o(1/N)$ and thus
$$
\E\left[\log(1+\eta)\right]=
-\frac{\rho}{N}\frac{3\rho+5\beta^2}{(\rho+\beta^2)^2}+o\left(\frac1N\right).
$$
Combining this with~\eqref{eq:lambda2eta} and subtracting~\eqref{lam1asy}, we get
\begin{align*}
\lambda_1-\lambda_2&=
\frac{\rho\,\left(\rho+2\beta^2\right)}{2N(\rho+\beta^2)^2}+o\left(\frac1N\right) 
\,.
\end{align*}
\end{proof}

\begin{figure}
  \begin{center} 
  \includegraphics[scale=0.3]{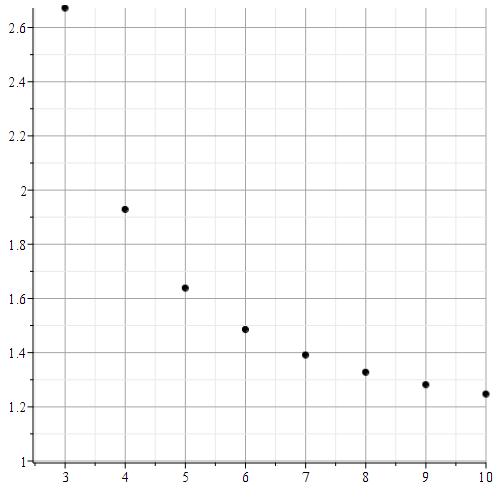}
  \caption{$\alpha_{\mathrm{cr}}$ as function of $N$ for $\beta = 0$}
  \label{figacrit}
  \end{center}
\end{figure}

\vskip 1cm

In the particular case $\beta = 0$, the expectation in~\eqref{E: maximal exponent} is straightforward to compute analytically, and evaluates (see \cite{CohenNewman}) to $$
\lambda_1(\alpha,0,N) = \frac{\log \frac{2\alpha}{N} +  \psi\left(\frac{N-1}{2}\right)}{2}\,.
$$
Trivially, $\lambda_1$ is an increasing function of $\alpha$ and the critical value of $\alpha$
is
$$
\alpha_{\mathrm{cr}}(0) = \frac{N}{2}\exp\left(-\psi\left(\frac{N-1}{2}\right)\right)\,,
$$
(see Figure~\ref{figacrit}). The other Lyapunov exponents are also known in the $\beta=0$ case (see~\cite{Newman}). For $k = 1, \dots, N-1$ we have 
$$
\lambda_k = \frac{\log \frac{2\alpha}{N} +  \psi\left(\frac{N-k}{2}\right)}{2}\,.
$$

For $\beta>0$ we have the following bound.
\begin{proposition}
For $\beta > 0$, we have 
$$
\alpha_\mathrm{cr}(\beta) < \frac{\alpha_\mathrm{cr}(0)}{(1-\beta)^2}
 =\frac{N\, e^{-\psi\left(\frac{N-1}{2}\right)}}{2(1-\beta)^2}\,.
$$
\end{proposition}
\begin{proof}
From Theorem~\ref{t1}, we see that $\lambda_1=0$ if and only if
$$
\phi_{\frac{N-1}{2}}
\left(\frac{N\beta^2}{2\rho}\right)
+\log \frac{2\rho}N =0
$$
where $\rho=\alpha(1-\beta)^2$. Let $\rho_\mathrm{cr}(\beta)=(1-\beta)^2\alpha_\mathrm{cr}(\beta)$ be the $\rho$ for which the displayed equality above holds for a given $\beta$.
Since $\phi_m(x)>\phi_m(0)\equiv \psi(m)$ by Lemma~\ref{lem:phi}, we get for $\beta>0$
$$
\log \frac{2\rho_\mathrm{cr}(\beta)}N =
-\phi_{\frac{N-1}{2}}
\left(\frac{N\beta^2}{2\rho_\mathrm{cr}(\beta)}\right)
<-\phi_{\frac{N-1}{2}}(0)=\log \frac{2\rho_\mathrm{cr}(0)}N
$$
yielding $\rho_\mathrm{cr}(\beta)<\rho_\mathrm{cr}(0)$.
\end{proof}

\vskip 1cm
\begin{lemma}\label{lemrhomonotone}
For $N\ge 3$ (i.e., $m\ge 1$) the function $\rho_\mathrm{cr}(\beta)$ is decreasing in~$\beta$.    
\end{lemma}
\begin{proof}
Indeed,
$\rho_\mathrm{cr}(\beta)$ satisfies
$$
\phi_m(z)+\log \rho_\mathrm{cr}(\beta)\equiv const
$$
where $z=N\beta^2/(2\rho_\mathrm{cr}(\beta))$. Differentiating w.r.t.\ $\beta$ we get
$$
\phi_m'(z)\left[\frac{N\beta}{\rho}-\frac{N\beta^2}{2\rho^2}\rho'\right]+\frac{\rho'}{\rho}=0
\quad
\Longrightarrow
\quad
\rho'(z\phi_m'(z)-1)= N\beta \phi_m'(z).
$$
But $\phi_m(\cdot)$ is an increasing function with $\phi_m'>0$, and by~\eqref{eq:phiprime} 
\begin{equation}\label{E: zphiprime bound}
z\phi_m'(z)=z e^{-z}\sum_{j=0}^\infty \frac{z^j}{j!(m+j)}
\le  e^{-z}\sum_{j=0}^\infty \frac{z^{j+1}}{j!(j+1)}
= e^{-z}\sum_{k=1}^\infty \frac{z^k}{k!}
=1-e^{-z}<1
\end{equation}
since $m\ge 1$, yielding $\rho'(\beta)<0$ for all $\beta\in[0,1)$. As a result
$$
\alpha_\mathrm{cr}(\beta)(1-\beta)^2<\alpha_\mathrm{cr}(\beta_0)(1-\beta_0)^2
$$
whenever $\beta>\beta_0\ge 0$ and $N\ge 3$.
\end{proof}
\begin{remark}
It seems from numerics that Lemma~\ref{lemrhomonotone} is true for $N=2$ as well, but we do not have a proof of this. \end{remark}

\begin{proposition}\label{P: alpha crit near 1}
For $N \ge 4$ fixed, we have
$$ \alpha_{\mathrm{cr}}(\beta) \sim \frac{2N}{N-3} \,\cdot\, \frac{1}{1-\beta} \quad \text{as $\beta \nearrow 1$.}$$
\end{proposition}
\begin{proof}
Make $z$ depend on $\beta$ by substituting $\alpha =\alpha_{\mathrm{cr}}(\beta)$ into $z = \frac{N \beta^2}{2\alpha(1-\beta)^2}$. From the equality 
$$
\lambda_1(\alpha_{\mathrm{cr}}(\beta),\beta,N) = \log \beta + \frac{1}{2}\left[\phi_{\frac{N-1}{2}}(z) - \log z\right]  = 0
$$ 
we find that $\log \beta$ is equal to a function of $z$. By the inequality~\eqref{E: zphiprime bound}, this function of $z$ is decreasing in $z$. Using Lemma~\ref{lem:phiasympt} we obtain
$$
0=2\log\beta+\frac{m-1}{z}(1 +o(1))
$$
so that we must have $z \to \infty$ as $\beta \nearrow 1$. Hence
$$
0=2\log\beta+\frac{m-1}{z}(1 + o(1)))
=2\log\beta+\frac{(N-3)/2}{\frac{N\beta^2}{2\alpha_{\mathrm{cr}}(\beta)(1-\beta)^2}}( 1 + o(1))
$$
so that 
$$
\alpha_\mathrm{cr}(\beta) \sim \frac{2N}{N-3}\cdot \frac{\log \beta^{-1}}{(1-\beta)^2} \sim \frac{2N}{N-3}\,.\,\frac{1}{1-\beta} \quad \text{as $\beta \nearrow 1$}
$$
when $N\ge 4$, as required. \end{proof}

\section{Analysis of Model~B}\label{SecB}

Recall that Model~B is defined by the It\^{o} form matrix stochastic differential equation~\eqref{E: SDE}:
$$ dZ(t) = -\gamma\overline{Z}(t) dt +   dB(t) T(t)$$
where
$$ \overline{Z}(t) =   \left(I_N- \frac{1}{N}\mathbf{1}_{N,N}\right)Z(t)$$
$$ 
\Cov^Z(t) = \frac{1}{N} \overline{Z}(t)^T\, \overline{Z}(t)\,,
$$
and $T(t)$ is the unique non-negative definite symmetric square root of $\Cov^Z(t)$. 
\begin{lemma}\label{L: SDE is well posed} The stochastic differential equation~\eqref{E: SDE} has a unique strong solution.\end{lemma}
\begin{proof}
We claim that the coefficients $- \gamma\overline{Z}(t) = -\gamma\left(Z(t) - \mathbf{1}_{N,1}\mu^Z(t)\right)$ and $T(t)$ are globally Lipschitz functions of the state $Z(t)$. The existence and uniqueness of a strong solution follows from this claim by a well-known theorem of It\^{o} (see, for example, \cite[Thm IV.3.1]{Ikeda} or \cite[Thms 5.2.5 and 5.2.9]{Karatzas}). 
Since all norms on a given finite-dimensional space of matrices are equivalent, it does not matter which norm we use to check the global Lipschitz condition. It is easiest to use the Frobenius norm. The Frobenius norm $\|A\|_2$ of an $m$~by~$n$ (real or complex) matrix $A$ is simply the square root of the sum of the squares of the absolute values of its entries.  

The absolute value {$|A|$} of a rectangular real or complex matrix $A$ is the unique non-negative definite Hermitian square root of $A^*A$, where $A^*$ denotes the Hermitian conjugate of $A$. Note that $T(t) = \frac{1}{N}\left|\overline{Z}(t)\right|$. We apply the Araki-Yamagami inequality (originally proven in \cite[Thm~1]{Araki} for operators on a Hilbert space, but see~\cite[Thm~VII.5.7, eqn~VII.39]{Bhatia} for a version that applies to arbitrary rectangular real or complex matrices). This says that the map that sends any matrix $A$ to its absolute value $|A|$ is Lipschitz with constant $\sqrt{2}$ with respect to the Frobenius norm:
\begin{equation}
\label{E: Araki-Yamagami inequality} 
\|\, |A| - |B|\,\|_2 \le \sqrt{2} \| A - B\|_2\,.
\end{equation}
This shows that the map $\overline{Z}(t) \mapsto T(t)$ is a $(\sqrt{2}/N)$-Lipschitz map from the space of $N$ by $d$ real matrices to the space of $d$ by $d$ real matrices, with respect to their Frobenius norms.
Since the map $Z(t) \mapsto \overline{Z}(t)$ is an orthogonal projection with respect to the Frobenius norm, it is $1$-Lipschitz, and hence the composition $Z(t) \mapsto T(t)$ is $(\sqrt{2}/N)$-Lipschitz.
\end{proof}
From the uniqueness statement, we may deduce the following corollary:
\begin{corollary}
The process $Z(\cdot)$ is affinely invariant:  if $M$ is any $d$ by $d$ real matrix and $v \in \mathbb{R}^d$ is any row vector, then $(Z(t)M + \mathbf{1}_{N,1}v)_{t \ge 0}$ is the unique strong solution of the SDE~\eqref{E: SDE} started from the initial condition $Z(0)M + \mathbf{1}_{N,1}v$, (with the same driving Brownian motions).
\end{corollary}

Note that $\overline{Z}(\cdot)$ satisfies an autonomous SDE, namely
$$d\overline{Z}(t)  = \left(I_N- \frac{1}{N}\mathbf{1}_{N,N}\right)d Z(t) = -\gamma\overline{Z}(t) dt + d\overline{B}(t)\,T(t)\,,  $$
where $\overline{B}(t) = \left(I_N- \frac{1}{N}\mathbf{1}_{N,N}\right)B(t)$.
We also remark that the infinitesimal covariation of the process $Z(\cdot)$ satisfies (in matrix form) $dZ(t)^T dZ(t) = \Cov^Z(t)\,dt$.

\subsection{Proof of Theorem~\ref{T: continuous time solution}}

Norris, Rogers, and Williams \cite{NRW} discuss
 the right-invariant Brownian motion $G$ on $GL(n,\mathbb{R})$, satisfying $G(0) = I$, and $$ \partial G = (\partial B)G$$ where $\partial$ denotes the Stratonovich differential, which is convenient here for the simple form of its chain rule, and $B$ is a standard $n^2$-dimensional Brownian motion, thought of as an $n$ by $n$ matrix of independent standard Brownian motions. The It\^{o} form of this matrix SDE is
$$dG = dB\, G + \frac{1}{2}dB\,dG = dB\, G + \frac{1}{2}G \,dt\,. $$

The right-invariance implies that for each $u> 0$, the process $(G(t+u)G(u)^{-1}: t \ge 0)$ is identical in law to $G$ and is independent of the process $(G(r): r \le u)$. Therefore, conditional on $G(u) = K$, the process $(G(t+u): t \ge 0)$ is identical in law to $(G(t)K: t \ge 0)$, which is a right-invariant Brownian motion started at $K$. The process $G(\cdot)$ almost surely exists for all time. The authors of \cite{NRW} note that $\left(G^{-1}\right)^T$ is identical in law to $G$, and that if $v(0) \in \mathbb{R}^n$ is any fixed vector then the process defined by $v(t): = G(t)v(0)$ satisfies the autonomous Stratonovich form SDE $$ \partial v(t) = (\partial B) v(t)\,.$$ They define $X(t) = G(t)G(t)^T$ and $Y(t) = G(t)^T G(t)$, and study the properties of the two processes~$X$ and~$Y$ in the space of positive-definite symmetric matrices. In particular, they show that each of $X(\cdot)$ and~$Y(\cdot)$ is a Markov process in its own generated filtration. Moreover, the law of $Y$ is invariant under similarity transformation $Y \mapsto K^T YK$, for an arbitrary fixed element $K \in O(n)$. (In \cite{NRW} this is stated for $K \in GL(n)$, but in fact orthogonality of $K$ is required to make the similarity transformation preserve the initial condition $G(0) = I$.)  

{In the proof of Theorem~\ref{T: continuous time solution} we will apply the results of \cite{NRW} taking $n = N-1$. We will also need the following simple result.}

\begin{lemma}\label{L: rotating the driving motion}
Let $W$ be an $N-1$ by $N-1$ matrix of independent standard Brownian motions and let $R$ be a continuous semimartingale adapted to the same filtration as $W$, taking its values in $N-1$ by $d$ real matrices, such that the columns of $R$ are almost surely orthogonal. Let $S$ be the solution of the matrix It\^{o} equation $dS = dW R$ with initial value $S(0) = \mathbf{0}_{N-1,d}$. Then $S$ is an $N-1$ by $d$ matrix of independent Brownian motions, adapted to the same filtration as $W$. 
\end{lemma}
\begin{proof}
 $S$ is a continuous matrix martingale, so by L\'evy's characterization of multidimensional Brownian motion it suffices to check the covariation:
 $$dS_{i,j} dS_{k,\ell}   =  \sum_{a =1}^d\sum_{b = 1}^d dW_{ia} R_{aj} dW_{kb} R_{b\ell}
  =  dt \sum_{a,b} \delta_{ik} \delta_{ab} R_{aj} R_{b\ell} =  \delta_{ik} \delta_{j\ell}\,dt.$$

\end{proof}
With this background in hand, we are in a position to prove Theorem~\ref{T: continuous time solution}.
\begin{proof}[Proof of Theorem~\ref{T: continuous time solution}]
Let the processes $\hat{W}$, $H$, $\overline{Z}$, $Z$, $F$ and $T$ be as constructed in the statement of Theorem~\ref{T: continuous time solution}.
We begin by showing that the process $VZ$, which takes its values in the space of $N-1$ by $d$ matrices, and is equal to $V\overline{Z}$, satisfies the following Stratonovich form SDE:
\begin{equation}\label{E: dagger}
\partial(VZ)(t) = - \gamma' VZ(t) \partial t + \frac{1}{\sqrt{N}} \partial W(t) VZ(t)\,,
\end{equation}
where $W$ is an $N-1$ by $N-1$ matrix of independent standard Brownian motions. This follows on applying the Leibniz rule to the defining equation for $\overline{Z}$, which we recall:
$$\overline{Z}(t) = e^{-\gamma' t} V^T G(t/N) V\overline{Z}(0)\,.$$
The final term in~\eqref{E: dagger} arises from setting $$W_t = \sqrt{N}\hat{W}_{t/N}\,$$ where $\hat{W}$ is the driving matrix of Brownian motions in~\eqref{E: H def}, which we use in the form $\partial G(t/N) = \partial \hat{W}_{t/N} G(t/N)$.  

We now convert~\eqref{E: dagger} into It\^{o} form:
\begin{equation}\label{E: double dagger}
d(VZ)(t) = -\gamma' VZ(t) dt + \frac{1}{\sqrt{N}} dW(t) VZ(t) + \frac{1}{2\sqrt{N}} dW(t)d(VZ)(t)\,.
\end{equation}
We may compute the covariation term $dW(t)d(VZ)(t)$ by computing the infinitesimal covariation of $dW$ with both sides of~\eqref{E: double dagger}, obtaining
$$ dW(t)\, d(VZ)(t) =  \frac{1}{\sqrt{N}} dW(t)\, dW(t) VZ(t) = \frac{1}{\sqrt{N}} I_{N-1} dt\, VZ(t) = \frac{1}{\sqrt{N}} VZ(t) dt\,.$$
Substituting this back into \eqref{E: double dagger} we obtain 
\begin{equation}\label{E: triple dagger}
d(VZ)(t) = -\left(\gamma'-\frac{1}{2N}\right) VZ(t) dt + \frac{1}{\sqrt{N}} dW(t) VZ(t)  = -\gamma VZ(t) dt + \frac{1}{\sqrt{N}}dW(t)VZ(t)\,.
\end{equation}
We see that the passage from Stratonovich to It\^{o} form merely requires a modification of the coefficient of the drift of $VZ$ towards $0$.

Since $VV^T = I_{N-1}$ and $V^T VZ(t) = \left(I_{N} - \frac{1}{N}\mathbf{1}_{N,N}\right) Z = \overline{Z}(t)$, we have
$$ N \mathrm{Cov}^Z(t) = \overline{Z}(t)^T \overline{Z}(t) = Z(t)^T V^T V V^T VZ(t) = Z(t)^T V^T V Z(t) = (VZ(t))^T VZ(t)\,.$$
Recall that $\overline{Z}$ was constructed as
$$ \overline{Z}(t) = e^{-\gamma' t} V^T G(t/N) V \overline{Z}(0)\,, $$
so we have
$$ VZ(t) = V\overline{Z}(t) = e^{-\gamma' t} VV^T G(t/N) VZ(0) = e^{-\gamma' t}  G(t/N) VZ(0)\,,$$ and hence
\begin{equation}\label{E: Cov Z expression} N \mathrm{Cov}^Z(t) = (VZ(t))^T VZ(t) = e^{-2\gamma' t} Z(0)^T V^T G(t/N)^T G(t/N) VZ(0)\,.
\end{equation}
Since $G(t/N) \in GL(N-1,\mathbb{R})$ for all $t$, we find that $G(t/N) VZ(0)$ has the same rank as $VZ(0)$, namely $d$, and hence $\mathrm{Cov}^Z(t)$ is strictly positive definite, and so $T(t)$ is non-singular. Since the positive definite square-root map is analytic on the space of {strictly} positive definite symmetric matrices, $T(t)$ is a continuous semimartingale, and {similarly, since matrix inversion is analytic away from the singular matrices, $T(t)^{-1}$ is also a continuous semimartingale}.
Define $$R(t) = VZ(t) T(t)^{-1}\,,$$ and observe that since $T(t)$ is a symmetric matrix, we have
$$ 
R(t)^T R(t) = N T(t)^{-1}(VZ(t))^T VZ(t) T(t)^{-1} = T(t)^{-1} T(t)^2 T(t)^{-1} = I_d\,,
$$ 
which is to say that $R(t)$ is an $N-1$ by $d$ matrix with orthogonal columns. Moreover, $R$ is a continuous semimartingale adapted to the same filtration as $W$.  Now we apply Lemma~\ref{L: rotating the driving motion} to obtain an $N-1$~by~$d$ matrix~$S$ of independent Brownian motions such that
$$ 
dS(t) = dW(t) R(t) = dW(t) VZ(t) T(t)^{-1}\,,
$$
and hence
$$ dS(t) T(t) = dW(t) VZ(t)\,.$$ This means we may rewrite~\eqref{E: triple dagger} as
\begin{equation}\label{E: quadruple dagger}
d(VZ)(t)  = -\gamma VZ(t) dt + \frac{1}{\sqrt{N}}dS(t)T(t)\,.
\end{equation}
Multiplying on the left by $V^T$ and using $VZ = V\overline{Z}$ we obtain
\begin{equation}\label{E: quintuple dagger}
d\overline{Z}(t)  = -\gamma \overline{Z}(t) dt + \frac{1}{\sqrt{N}}d(V^TS)(t)T(t)\,.
\end{equation}
Now define 
\begin{equation}\label{E: B def} B(t) := V^T S(t) + \frac{1}{\sqrt{N}}\mathbf{1}_{N,1} F(t)\,.\end{equation}
Since the columns of $V^T$ are orthonormal and orthogonal to the unit column vector $\frac{1}{\sqrt{N}}\mathbf{1}_{N,1}$, we see by a calculation similar to the proof of Lemma~\ref{L: rotating the driving motion} that $B$ is an $N$ by $d$ matrix of independent standard Brownian motions. 

Recall that $Z$ is constructed from $Z(0)$ and the process $\overline{Z}$ by solving 
\begin{equation}\label{E: Z from Zbar} dZ(t) = d\overline{Z}(t) + \frac{1}{\sqrt{N}} \mathbf{1}_{N,1} dF(t) T(t)\,.\end{equation}
Combining~\eqref{E: quintuple dagger},~\eqref{E: B def}, and~\eqref{E: Z from Zbar}, we obtain equation~\eqref{E: SDE}. To check that we have a strong solution, note that $F$ and $S$ may be recovered from $B$, so it is apparent from~\eqref{E: quintuple dagger} and~\eqref{E: Z from Zbar} that $Z(t)$ is a measurable function of $(B_s)_{0 \le s \le t}$. This completes the proof of Theorem~\ref{T: continuous time solution}.
\end{proof}
\begin{proof}[Proof of Corollary~\ref{C: critical gamma}]
When $\gamma > \frac{1}{2} - \frac{3}{2N}$, we have $2\gamma' > \frac{N-2}{N}$. Now, $\frac{N-2}{N}$ is the leading Lyapunov exponent of the process $Y(t/N) = G(t/N)^T G(t/N)$ (see \cite{NRW}). Hence  by equation~\eqref{E: Cov Z expression} we have $\overline{Z}(t)^T \overline{Z}(t) \to 0$ a.s., which implies that $\overline{Z}(t) \to 0$ a.s., and in fact  $T(t)$ decays exponentially, which implies that $Z(t) -\overline{Z}(t)$ a.s.~converges as $t \to \infty$.

On the other hand, suppose that $\gamma < \frac{1}{2} - \frac{3}{2N}$. Then a.s.~the leading eigenvalue of $\mathrm{Cov}^Z(t)$ grows exponentially as $t \to \infty$, so $\overline{Z}(t)$ and $Z(t)$ must both be unbounded as $t \to \infty$.  
\end{proof}
We conclude this section with a few remarks about the process $\mathrm{Cov}^Z$. Almost surely the rank of $\Cov^Z(t)$ is equal for all $t$ to the rank of $\Cov^Z(0)$, and the minimal affine space containing $\{Z_1(t), \dots, Z_N(t)\}$ is $\mathbb{R}^d$ for all $t$. This follows from the fact that $G(t)$ stays in $GL(N-1,\mathbb{R})$ for all time, almost surely.

Finally, we will prove Proposition~\ref{P: covariance is Markov in model B}, which states that $\mathrm{Cov}^Z(t)$ is a Markov process in its own filtration.
\begin{proof}[Proof of Proposition~\ref{P: covariance is Markov in model B}]
We use~\eqref{E: triple dagger} to compute 
$$ dW(t)\,dVZ(t) = \frac{1}{\sqrt{N}} VZ(t) \,dt$$ and
$$ dW(t)^T dVZ(t) = \frac{1}{\sqrt{N}} dW(t)^T dW(t) VZ(t) = \frac{N-1}{\sqrt{N}}VZ(t)\,dt\,,$$
hence
\begin{align*}
(d(VZ)(t))^T d(VZ)(t) &= \frac{1}{\sqrt{N}}(d(VZ)(t))^T dW(t) VZ(t) 
\\& = \frac{N-1}{N} (VZ(t))^T VZ(t) = (N-1)\mathrm{Cov}^Z(t)\,.
\end{align*}
Taking the It\^{o} differential of both sides of the equation $N\mathrm{Cov}^Z(t) = VZ(t)^T VZ(t)$, we find
\begin{eqnarray*} N d \mathrm{Cov}^Z(t) & = & d(VZ(t))^T VZ(t) + (VZ(t))^T dVZ(t) + dVZ(t)^T \,dVZ(t)\\
& = & (N-1 -2\gamma N) \mathrm{Cov}^Z(t) + \frac{1}{\sqrt{N}} (VZ(t))^T(dW(t)^T + dW(t))VZ(t) 
\end{eqnarray*}
We now use the (canonical) decomposition  $VZ(t) = R(t)T(t)$, and another application of Lemma~\ref{L: rotating the driving motion}, to express this as
$$N d\mathrm{Cov}^Z(t) = (N-1 - 2\gamma N)\mathrm{Cov}^Z(t) + \frac{1}{\sqrt{N}} T(t) (dS(t)^T R(t) + R(t)^TdS(t))T(t)\,, $$
where $S(t)$ is some $N-1$ by $d$ matrix of independent Brownian motions adapted to the same filtration as~$W$. Another calculation similar to the proof of Lemma~\ref{L: rotating the driving motion} shows that because $R(t)$ has orthogonal columns, we may replace $dS(t)^TR(t)$ by a $d$ by $d$ matrix $U$ of independent standard Brownian motions adapted to the same filtration as $W$. We obtain
$$N d\mathrm{Cov}^Z(t) = (N-1 - 2\gamma N)\mathrm{Cov}^Z(t) + \frac{1}{\sqrt{N}} T(t) (dU(t)^T + dU(t))T(t)\,, $$
Since $T(t)$ is a function of $\mathrm{Cov}^Z(t)$, this implies that $\left(\mathrm{Cov}^Z(t)\right)_{t \ge 0}$ is a Markov process in its own generated filtration. This result is similar to the result in \cite{NRW} that the process $Y(t) = G(t)^T G(t)$ is a Markov process in its own filtration. However, we were not able to deduce Proposition~\ref{P: covariance is Markov in model B} directly from the Markovian nature of~$Y$.
\end{proof}
\begin{remark} With extra work we believe that it should be possible to show that the mapping from $Z$ to $\mathrm{Cov}^Z$ is a strong lumping, which is to say that for any time $s > 0$, conditional on $\mathrm{Cov}^Z(s)$, the future process $\left(\mathrm{Cov}^Z(t)\right)_{t \ge s}$ and $Z(s)$ are independent. {We expect this by analogy with Model A, which does have the corresponding strong lumping property, as we explained after the statement of Proposition~\ref{prop: Covs are Markov}.}
\end{remark}

\section{Appendix}
%First, we will prove an auxiliary fact. 
Recall the definition of the function $\phi_m$ in equation~\eqref{def of phi}.
\begin{lemma}\label{lem:phi}
For all $m>0$ we have
\begin{align*}
\phi_m(x) &=\psi(m)
-\sum_{k=1}^\infty \frac{(-x)^k}{k\cdot m(m+1)\dots(m+k-1)}
=
\psi(m)+\int_0^1 \frac{(1-e^{-xs})(1-s)^{m-1}}{s}\mathrm{d}s.
\end{align*}
\end{lemma}
\begin{proof}
Differentiating the expression in~\eqref{def of phi}, and using the fact that $\psi(y+1)-\psi(y)=1/y$, we get
\begin{align}\label{eq:phiprime}
\begin{split}
\phi_m'(x)&=-e^{-x}\psi(m)-
e^{-x}\sum_{j=1}^\infty\psi(j+m)\left[\frac{x^j}{j!}-\frac{x^{j-1}}{(j-1)!}\right]
\\ &=
e^{-x}\left[ -\psi(m)+\psi(m+1)\right]
+e^{-x}\sum_{j=1}^\infty\frac{x^j}{j!}\left[\psi(m+j+1)-\psi(m+j)\right]
\\ &
=e^{-x}\left[\frac1m
+\sum_{j=1}^\infty\frac{x^j}{j!(m+j)}
\right]
=e^{-x}\sum_{j=0}^\infty\frac{x^j}{j!(m+j)}
=\mathbb{E}\left(\frac1{m+\xi}\right)
\end{split}\end{align}
where $\xi \sim \mathrm{Poi}(x)$. Note also that
$$
\left(x^m e^x \phi_m'(x)\right)'=\sum_{j=0}^\infty \frac{x^{j+m-1}}{j!}=x^{m-1} e^x
$$
so that $\phi_m(x)$ satisfies the second-order linear differential equation
$$
x \phi_m''+(x+m)\phi_m'=1;\quad\phi_m(0)=\psi(m); \quad \phi_m'(0)=\frac1m.
$$
The expression for $\phi_m'(x)$ can be simplified
further:
\begin{align*}
\phi_m'(x)&=\frac{e^{-x}}m\left[1
+\sum_{j=1}^\infty\frac{x^j}{(j-1)!} \left(\frac1j-\frac1{j+m}\right)
\right]=
\frac{e^{-x}}m\left[\sum_{j=0}^\infty\frac{x^j}{j!} 
-x\sum_{j=1}^\infty\frac{x^{j-1}}{(j-1)!(j+m)} 
\right]
\\ &=
\frac{e^{-x}}m\left[e^x
-x\sum_{j=0}^\infty\frac{x^{j}}{j!(j+m+1)} 
\right]
=\frac1m -\frac{x}m\phi'_{m+1}(x).
\end{align*}
Iterating this formula $\ell$ times, we obtain that
$$
\phi_m'(x)=\frac1m-\frac{x}{m(m+1)}
+\frac{x^2}{m(m+1)(m+2)}-\dots  \pm \frac{x^\ell}{m(m+1)\dots(m+\ell-1)}\phi'_{m+\ell}(x).
$$
Noting that $0\le \phi_m'(x)\le 1$ for all $m$ and letting $\ell\to\infty$ we obtain 
\begin{align}\label{phiT}
\phi_m'(x)=
\sum_{j=0}^\infty \frac{(-x)^j}{m(m+1)\dots(m+j)}.
\end{align}
Recalling that $\phi_m(0)=\psi(m)$, we obtain $$\phi_m(x)=\psi(m)-\sum_{j=0}^\infty \frac{(-x)^{j+1}}{(j+1)\cdot m(m+1)\dots(m+j)}.$$
To get the second expression for $\phi_m(x)$, note that
$$
\frac{1-e^{-xs}}{s}=x-\frac{x^2s}{2!}+\frac{x^3s^2}{3!}-\dots=-\sum_{k=1}^\infty \frac{(-x)^ks^{k-1}}{k!}
$$
and thus
\begin{align*}
\int_0^1 \frac{(1-e^{-xs})(1-s)^{m-1}}{s}\mathrm{d}s
&=-\int_0^1 \left(\sum_{k=1}^\infty \frac{(-x)^ks^{k-1}}{k!}\right)(1-s)^{m-1}\mathrm{d}s
\\ &
=-\sum_{k=1}^\infty \frac{(-x)^k}{k\,\Gamma(k)}\int_0^1 s^{k-1}(1-s)^{m-1}\mathrm{d}s
\\ &=-
\sum_{k=1}^\infty \frac{(-x)^k}{k\,\Gamma(k)} \frac{\Gamma(k)\,\Gamma(m)}{\Gamma(m+k)}
=-
\sum_{k=1}^\infty \frac{(-x)^k}{k\cdot m(m+1)\dots(m+k-1)}.
\end{align*}
\end{proof}

\begin{proof}[Proof of Proposition~\ref{prop:Nodd}]
In case of an integer $m=(N-1)/2$, the expression obtained in~\eqref{phiT}  can be further simplified to 
\begin{align*}
\phi_m'(x)=\frac{(m-1)!}{(-x)^m}\left[
\sum_{i=m}^{\infty}\frac{(-x)^i}{i!}
\right]=\frac{(m-1)!}{(-x)^m}\left[
e^{-x}-\sum_{i=0}^{m-1}\frac{(-x)^i}{i!}
\right]\,.
\end{align*}
Let
$$
A=\sum_{i=1}^{m-1} \frac{(i-1)!}{(-x)^i} e^{-x}, \quad
B=\sum_{i=1}^{m-1} \frac{(i-1)!}{(-x)^i}\binom{m-1}{i},
\quad C=\log (x)+\mathrm{Ei}_1(x).
$$
Then
\begin{align*}
\frac{d}{dx}A
&=\frac{e^{-x}}{x}+\frac{(m-1)!e^{-x}}{(-x)^m}
=\frac{e^{-x}}{x}+\frac{(m-1)!}{(-x)^m}\sum_{j=0}^\infty \frac{(-x)^j}{j!},
\\
\frac{d}{dx}B
&=\frac{(m-1)!}{(-x)^m}\left[1-x+\frac{x^2}{2!}-....+\frac{(-x)^{m-2}}{(m-2)!}\right],
\\
\frac{d}{dx}C&=\frac{1-e^{-x}}{x}
=\frac{(m-1)!}{(-x)^m}\cdot\left[-\frac{(-x)^{m-1}}{(m-1!)}\right] -\frac{e^{-x}}{x}
\end{align*}
hence $\displaystyle \frac{d}{dx}(A-B+C)=\sum_{j=m}^\infty \frac{(m-1)!}{j!} (-x)^{j-m}=\phi_m'(x)$ according to~\eqref{phiT}. Hence, $\phi_m(x)=A-B+C+\mathrm{const}$. Now we will show that this constant is, in fact, zero. Indeed,
$$
\lim_{x\downarrow 0} C=\lim_{x\downarrow 0} \left(\log(x)+\mathrm{Ei}_1(x)\right)=-\gamma,
$$
and
\begin{align*}
A&=\sum_{i=1}^{m-1} \frac{(i-1)!}{(-x)^i}
\sum_{k=0}^i \frac{(-x)^k}{k!}+O(x)
=\sum_{i=1}^{m-1}
\sum_{k=0}^{i-1} \frac{(i-1)!}{(-x)^{i-k}k!} 
+\sum_{i=1}^{m-1}\frac{(i-1)!}{i!}
+O(x)
\\ &
=\sum_{i=1}^{m-1}\sum_{j=1}^{i} \frac{(i-1)!}{(-x)^{j}(i-j)!} 
 +\sum_{i=1}^{m-1}\frac{1}{i}
+O(x)
=\sum_{j=1}^{m-1}\frac{1}{(-x)^{j}}\sum_{i=j}^{m-1} \frac{(i-1)!}{(i-j)!} 
 +\sum_{i=1}^{m-1}\frac{1}{i}
+O(x)
\\ &
=\sum_{j=1}^{m-1}\frac{(j-1)!}{(-x)^{j}}\sum_{i=j}^{m-1} \binom{i-1}{j-1}
 +\sum_{i=1}^{m-1}\frac{1}{i}
+O(x)
=\sum_{j=1}^{m-1}\frac{(j-1)!}{(-x)^{j}} \binom{m-1}{j}
 +\sum_{i=1}^{m-1}\frac{1}{i}
+O(x)
\\ &
=B+\sum_{i=1}^{m-1}\frac 1i +O(x)
\end{align*}
(using the formula from 24.1.1 in Abramowitz and Stegun~\cite{AS} in the penultimate equality)  resulting in
$$
\lim_{x\downarrow 0} (A-B)=\sum_{i=1}^{m-1}\frac 1i,
$$
So, since $\phi_m(0)=\psi(m)=-\gamma+1+\frac12+\dots+\frac1{m-1}$,
$$
\phi_m(x)=A-B+C=\log(x)+\mathrm{Ei}_1(x)+\sum_{i=1}^{m-1} \frac{(i-1)!}{(-x)^i}
\left[e^{-x}-\binom{m-1}{i}\right],\qquad  x>0,
$$
which together with Theorem~\ref{t1} yields the statement of the claim.
\end{proof}

\begin{remark}
Let  $z=\frac{N\beta^2}{2\alpha (1-\beta)^2}$.
From Proposition~\ref{prop:Nodd} we obtain
\begin{align*}
\lambda_1(\alpha,\beta,N)=\log\beta
+\frac{\mathrm{Ei}_1 \left(z\right)}2 +
\begin{cases}
\displaystyle
0, &\text{when }N=3, \\[2mm]
\displaystyle
\frac{1-e^{-z}}{2z}, &\text{when }N=5, \\[2mm]
\displaystyle
\frac{(z-1)(2-e^{-z})}{ 2z^2}+\frac1{2z^2}, &\text{when }N=7, \\[2mm] 
\displaystyle
\frac{(z^2-z+2)(3-e^{-z})}{2z^3}-\frac2{z^3}, &\text{when }N=9.
\end{cases}
\end{align*}
\end{remark}

\vspace{10mm}

\begin{lemma}\label{lem:phiasympt}
If $m\ge 1$, then as $x\to\infty$
$$
\phi_m(x)=\log x+\frac{m-1}{x}+o\left(\frac1x\right).
$$
\end{lemma}
\begin{proof}
From the integral representation 
\begin{align*}
\phi_m(x) -\psi(m)&=\int_0^1\frac{(1-e^{-xs})(1-s)^{m-1}}{s}\mathrm{d}s
\\ &=
\int_0^1 \frac{1-e^{-xs}}{s}\mathrm{d}s
-(m-1)\int_0^1 1-e^{-xs}\mathrm{d}s
\\ &
+\int_0^1 \frac{1-e^{-xs}}{s}\left[(1-s)^{m-1}-1+(m-1)s\right]\mathrm{d}s
\\ &=\gamma+\log x+\mathrm{Ei}_1(x)-(m-1)+\frac{(m-1)(1-e^{-x})}x+I
\end{align*}
where
\begin{align*}
I&=\int_0^1 \frac{1-e^{-xs}}{s}\left[(1-s)^{m-1}-1+(m-1)s\right]\mathrm{d}s
\\ &=
\int_0^1 \sum_{k=2}^{m-1} (-1)^k\binom{m-1}{k}s^{k-1}\mathrm{d}s
-
\int_0^1 \frac{e^{-xs}}{s}\left[(1-s)^{m-1}-1+(m-1)s\right]\mathrm{d}s
\\ &=
\sum_{k=2}^{m-1} \frac{(-1)^k}k \binom{m-1}{k}
+ R
=(m-1)-\psi(m)-\gamma+ R
\end{align*}
(from the Newton series for the digamma function) where
$$
R:=\int_0^1 \frac{e^{-xs}g(s)}s\mathrm{d}s,
\qquad g(s):=(1-s)^{m-1}-1+(m-1)s.
$$
Note that $\sup_{s\in[0,1]} |g(s)|\le m-1$ and by Peano's formula 
$$
g(s)=(m-1)(m-2)(1-\theta(s))^{m-3}\frac{s^2}{2!}.
$$
for some  $0\le \theta(s)\le s$. Hence for $0\le s\le 1/2$
we have $|g(s)|\le C_m s^2 $ for some $C_m>0$.
Consequently,
$$
|R|\le C_m\int_0^{1/2} s^2\,e^{-xs} \mathrm{d}s
+2(m-1)\int_{1/2}^1 e^{-xs}\mathrm{d}s
\le 
 C_m \int_0^\infty s\, e^{-xs} \mathrm{d}s+(m-1)e^{-x/2}=\frac1{x^2}
+o(x^{-2})
$$
yielding
\begin{align*}  
\phi_m(x)-\log x&=\psi(m)+\gamma+ \mathrm{Ei}_1(x)-(m-1)+\frac{m-1}{x}+I
\\ &
=\mathrm{Ei}_1(x)+\frac{m-1}{x}+O\left(\frac1{x^2}\right)=\frac{m-1}{x}+o(1/x)
\end{align*}   
\end{proof}

\section*{Acknowledgment}
The research of S.V.\ is partially supported by the Swedish Science Foundation grant VR 2019-04173. S.V.~would like to acknowledge the hospitality of the University of Bristol during his visits to Bristol. The research of E.C.\ is supported by the Heilbronn Institute for Mathematical Research.  
%{We are also thankful for the anonymous referees for their helpful comments.}

\end{document}